\documentclass[10pt]{amsart}
\usepackage{amsmath, amsthm, amssymb, euscript, enumerate}
\usepackage{pxfonts, graphicx, subfigure, epsfig}
\usepackage{mathtools, appendix, accents, color}
\usepackage[pdftex]{hyperref}
\mathtoolsset{showonlyrefs}
\usepackage{esvect}

\newcommand{\R}{{\mathbb R}}
\newcommand{\sS}{{\mathbb S}}

\newcommand{\Z}{{\mathbb Z}}

\newcommand{\e}{\varepsilon}

\newcommand{\p}{\partial}

\newcommand{\ra}{\rightarrow}

\newcommand{\norm}[1]{\left\Arrowvert {#1}\right\Arrowvert}

\newcommand{\osc}{\operatornamewithlimits{osc}}

\newcommand{\tr}{\operatorname{tr}}
\newcommand{\Lip}{\operatorname{Lip}}

\theoremstyle{plain}
\newtheorem{theorem}{Theorem}[section]

\newtheorem{lemma}[theorem]{Lemma}

\theoremstyle{definition}

\theoremstyle{remark}

\newtheorem{remark}[theorem]{Remark}

\numberwithin{equation}{section}

\title[Higher Order Convergence Rates: Hamilton-Jacobi Equations]{Higher Order Convergence Rates in Theory of Homogenization III: viscous Hamilton-Jacobi Equations}
\author{Sunghan Kim}
\address{Department of Mathematical Sciences, Seoul National University, Seoul 08826, Korea}
\email{sunghan290@snu.ac.kr}
\author{Ki-Ahm Lee}
\address{Department of Mathematical Sciences, Seoul National University, Seoul 08826, Korea
\& Center for Mathematical Challenges, Korea Institute for Advanced Study, Seoul 02455, Korea}
\email{kiahm@snu.ac.kr}
\thanks {S. Kim has been supported by National Research Foundation of Korea (NRF) grant funded by the Korean government (NRF-2014-Fostering Core Leaders of the Future Basic Science Program). K.-A. Lee has been supported by NRF grant funded by the Korean government (MSIP) (No. 2017R1A2A2A05001376). K.-A. Lee also holds a joint appointment with the Research Institute of Mathematics of Seoul National University.}

\begin{document}

\maketitle
\begin{abstract} 
In this paper, we establish the higher order convergence rates in periodic homogenization of viscous Hamilton-Jacobi equations, which is convex and grows quadratically in the gradient variable. We observe that although the nonlinear structure governs the first order approximation, the nonlinear effect is absorbed as an external source term of a linear equation in the second and higher order approximation. Moreover, we find that the geometric shape of the initial data has to be chosen carefully according to the effective Hamiltonian, in order to achieve the higher order convergence rates. 
\end{abstract}

\tableofcontents

%%%%%%%%%%%%%%%%%%%%%%%%%%%%%%%%%%%%%%%%%%%%%%%%%%%%%%%%%%%%%%%%%%%%%%%%%%%%%%%%%

\section{Introduction}\label{section:intro}

This paper concerns the higher order convergence rates of the homogenization of viscous Hamilton-Jacobi equations. The model problem is of the form, 
\begin{equation}\label{eq:HJ}
\begin{dcases}
u_t^\e - \e \tr\left( A\left(\frac{x}{\e}\right) D^2 u^\e\right) + H \left( Du^\e, \frac{x}{\e} \right) = 0 & \text{in }\R^n\times(0,\infty),\\
u^\e = g & \text{on }\R^n\times\{t=0\}.
\end{dcases}
\end{equation}
Here the diffusion matrix $A$ is periodic and uniformly elliptic, and the Hamiltonian $H$ is periodic in the spatial variable while it is convex and grows quadratically in the gradient variable. The initial data $g$ will be chosen to have smooth solutions for the effective Hamilton-Jacobi equation. At the end of this paper, we shall extend the result to the fully nonlinear, viscous Hamilton-Jacobi equation in the form of 
\begin{equation}\label{eq:HJ-nl}
\begin{dcases}
u_t^\e + H \left( \e D^2 u^\e, Du^\e,\frac{x}{\e}\right) = 0 & \text{in }\R^n\times(0,\infty),\\
u^\e = g & \text{on }\R^n\times\{t=0\}.
\end{dcases}
\end{equation}

This paper is in the sequel of the authors' previous works \cite{KL1} and \cite{KL2}, where the higher order convergence rates were achieved in the periodic homogenization of fully nonlinear, uniformly elliptic and parabolic, second order PDEs. We found it interesting in the previous works that even if we begin with a nonlinear PDE at the first order approximation, we no longer encounter such a nonlinear structure in the second and the higher order approximations. Instead, we always obtain a linear PDE with an external source term, which can be interpreted as the nonlinear effect coming from the error that is left undetected in the previous step of the approximation. 

The previous papers were concerned of uniformly elliptic (or parabolic) PDEs that are nonlinear in the second order derivatives, where the nonlinear perturbation is still made in the same order of the linear structure. A key difference in the current paper is that we impose a nonlinear structure (in the gradient term) that has quadratic growth at the infinity, so that this nonlinearity cannot be attained by order 1 perturbations of a linear structure. We believe that the quadratic growth condition can be generalized to superlinear growth condition, only if the solution of the corresponding effective problem is smooth enough.

Another interesting fact we found in studying Hamilton-Jacobi equations is that the geometric shape of the initial data turns out to play an important role in achieving higher order convergence rates. In particular, what we observe in this paper is that the geometric shape of the initial data has to be selected according to the nonlinear structure of the effective Hamiltonian, which to the best of our knowledge has not yet been addressed in any existing literature. The main reason for this requirement is to ensure the solution of the effective problem to be sufficiently smooth such that one can proceed with the approximation as much as one desires. 

In this paper, we establish higher order convergence rates when the initial data is convex, while the Hamiltonian is convex. However, a natural question is if one can generalize one of these structure conditions, which seems to be an interesting yet challenging problem. We shall come back to this in the forthcoming paper.  

The periodic homogenization of (viscous) Hamilton-Jacobi equations is by now considered to be standard, and one may consult the classical materials \cite{LPV} and \cite{E} for a rigorous justification. For the notion of viscosity solutions and the standard theory in this framework we refer to \cite{CIL} and \cite{CC}. 

For the recent development in the rate of convergence in periodic homogenization of (viscous) Hamilton-Jacobi equations, we refer to \cite{CDI}, \cite{CCDG}, \cite{M}, \cite{MT}, and the references therein. Nevertheless, this is the first work on the higher order convergence rates in the regime of (viscous) Hamilton-Jacobi equations. For the higher order convergence rates for other type of equations, we refer to \cite{KL1}, \cite{KL2} and the references therein. 

The paper is organized as follows. In Section \ref{section:assump}, we introduce basic notation used throughout this paper, and list up the standing assumptions regarding the main problem \eqref{eq:HJ}. From Section \ref{section:prelim} to Section \ref{section:cor}, we are concerned with the homogenization problem of \eqref{eq:HJ}. In Section \ref{section:prelim}, we summarize some standard results on the cell problem and the effective Hamiltonian. In Section \ref{section:reg}, we establish the regularity theory of interior correctors in the slow variable. Based on this regularity theory, we construct the higher order interior correctors in Section \ref{section:cor} and prove Theorem \ref{theorem:cr}, which is the first main result. Finally in Section \ref{section:nl}, we generalize this result to the homogenization of \eqref{eq:HJ-nl}, and prove Theorem \ref{theorem:cr-nl}, which is the second main result.

%%%%%%%%%%%%%%%%%%%%%%%%%%%%%%%%%%%%%%%%%%%%%%%%%%%%%%%%%%%%%%%%%%%%%%%%%%%%%%%%%

\section{Basic Notation and Standing Assumption}\label{section:assump}

Throughout the paper, we set $n\geq 1$ to be the spatial dimension. The parameters $\lambda$, $\Lambda$, $\alpha$, $\alpha'$, $\beta$, $\beta'$, $K$, $L$, and $\bar\mu$ will be fixed positive constants, unless stated otherwise. By $\Z^n$ we denote the space of $n$-tuple of integers. By $\sS^n$ we denote the space of all symmetric $n\times n$ matrices. By $C^\infty(X; C^{k,\mu}(Y))$, we denote the space of functions $f=f(x,y)$ on $X\times Y$ such that $f(\cdot,y)\in C^\infty(X)$ for all $y\in Y$ and $\{D_x^k f(x,\cdot)\}_{x\in X}$ is uniformly bounded in $C^{k,\mu}(Y)$. 

From Section \ref{section:prelim} to Section \ref{section:cor}, we study the higher order convergence rates in homogenization of \eqref{eq:HJ}. Throughout these sections, we assume that the diffusion matrix $A$ satisfies the following, for any $y\in\R^n$. 
\begin{enumerate}[(i)]
\item $A$ is periodic:
\begin{equation}\label{eq:A-peri}
A(y+k) = A(y).
\end{equation}
\item $A$ is uniformly elliptic:
\begin{equation}\label{eq:A-ellip}
\lambda I \leq A(y) \leq \Lambda I.
\end{equation}
\item $A\in C^{0,1}(\R^n)$ and
\begin{equation}\label{eq:A-C01}
\norm{A}_{C^{0,1}(\R^n)} \leq K.
\end{equation}
\end{enumerate}
On the other hand, we shall assume that the Hamiltonian $H$ verifies the following, for any $(p,y)\in\R^n\times\R^n$. 
\begin{enumerate}[(i)]
\item $H$ is periodic in $y$:  
\begin{equation}\label{eq:H-peri}
H(p,y+k) = H(p,y),
\end{equation}
for any $k\in\Z^n$.
\item $H$ has quadratic growth in $p$: 
\begin{equation}\label{eq:H-quad}
\alpha |p|^2 - \alpha' \leq H(p,y) \leq \beta|p|^2 + \beta'.
\end{equation}
\item $H$ is convex in $p$: 
\begin{equation}\label{eq:H-convex}
H(tp + (1-t)q,y) \leq t H(p,y) + (1-t) H(q,y),
\end{equation}
for any $0\leq t\leq 1$ and any $q\in\R^n$. 
\item $H\in C^\infty(\R^n;C^{0,1}(\R^n))$ and  
\begin{equation}\label{eq:H-Ck-C01}
\norm{D_p^k H(p,\cdot)}_{C^{0,1}(\R^n)} \leq K \left(1+|p|^{(2-k)_+}\right),
\end{equation}
for any nonnegative integer $k$.
\end{enumerate}

The assumptions on the initial data $g$ will be given in the beginning of Section \ref{section:cor}, since we need to derive the effective Hamiltonian beforehand. On the other hand, the structure conditions for \eqref{eq:HJ-nl} will be given in the beginning of Section \ref{section:nl}. 

%%%%%%%%%%%%%%%%%%%%%%%%%%%%%%%%%%%%%%%%%%%%%%%%%%%%%%%%%%%%%%%%%%%%%%%%%%%%%%%%%

\section{Preliminaries}\label{section:prelim}

Let us begin with the well-known cell problem for our model equation \eqref{eq:HJ}, stated as below. This lemma is by now considered to be standard (for instance, see \cite{E1} and \cite{E2}), since the diffusion coefficient $A$ is uniformly elliptic and the Hamiltonian $H$ is convex. Nevertheless, we shall present a proof for the reader's convenience.

\begin{lemma}\label{lemma:cell} For each $p\in\R^n$, there exists a unique real number, $\gamma$, for which the following PDE,
\begin{equation}\label{eq:cell}
-\tr (A(y) D^2w) + H (D w + p,y) = \gamma \quad\text{in }\R^n,
\end{equation}
has a periodic viscosity solution $w\in C^{2,\mu}(\R^n)$ for any $0<\mu<1$. Moreover, we have
\begin{equation}\label{eq:cell-Linf}
\alpha |p|^2 - \alpha' \leq \gamma \leq \beta |p|^2 + \beta'.
\end{equation}
Furthermore, a periodic solution $w$ of \eqref{eq:cell} is unique up to an additive constant, and satisfies
\begin{equation}\label{eq:cell-C2a}
(1+ |p|) \left(\norm{ w - w(0) }_{L^\infty(\R^n)} + \norm{ Dw}_{L^\infty(\R^n)} \right) + \norm{D^2 w}_{C^\mu(\R^n)} \leq C( 1 + |p|^2),
\end{equation}
where $C>0$ depends only on $n$, $\lambda$, $\Lambda$, $\alpha$, $\alpha'$, $\beta$, $\beta'$, $K$ and $\mu$. 
\end{lemma}

\begin{proof} Throughout the proof, $C$ will denote a positive, generic constant that depends at most on $n$, $\lambda$, $\Lambda$, $\alpha$, $\alpha'$, $\beta$, $\beta'$, $K$ and $\mu$, unless stated otherwise. Moreover, we shall fix $0<\mu<1$. 

Let $p\in\R^n$ be given. We know {\it a priori} that periodic viscosity solutions of \eqref{eq:cell}, if any, are unique up to an additive constant. Suppose that $w'$ is another periodic viscosity solution of \eqref{eq:cell}. Then $v = w-w'$ satisfies the following linearized equation,
\begin{equation*}
- \tr(A(y) D^2 v) + B(y) \cdot Dv = 0 \quad \text{in }\R^n,
\end{equation*}
where $B(y) = \int_0^1 D_p H( t D_y w + (1-t) D_y w' + p,y) dt$. Now that $v$ is bounded, we deduce from the Liouville theorem that $v$ is a constant function on $\R^n$. 

Henceforth, we prove the existence of a unique real number, $\gamma$, such that the cell problem \eqref{eq:cell} admits a periodic viscosity solution. The existence is proved by considering the following approximation problem,
\begin{equation}\label{eq:wd-pde}
- \tr(A(y) D^2 w^\delta) + H(D w^\delta + p, y) + \delta w^\delta = 0 \quad\text{in }\R^n,
\end{equation} 
for each $\delta>0$. Due to \eqref{eq:H-quad}, we know that $-\delta (\alpha |p|^2 - \alpha')$ and $-\delta (\beta|p|^2+\beta')$ are a supersolution and, respectively, a subsolution of \eqref{eq:wd-pde}. Thus, the comparison principle yields a unique viscosity solution, $w^\delta$, of \eqref{eq:wd-pde}, satisfying 
\begin{equation}\label{eq:dwd-Linf}
-\beta |p|^2 - \beta' \leq \delta w^\delta  \leq -\alpha |p|^2 + \alpha',
\end{equation}
on $\R^n$. The uniqueness of $w^\delta$ implies its periodicity, that is, $w^\delta (y+k) = w^\delta(y)$ for all $y\in\R^n$ and all $k\in\Z^n$. 

Let us remark here that  $w^\delta\in C^{0,1}(\R^n)$ and 
\begin{equation}\label{eq:wd-osc}
\norm{D w^\delta}_{L^\infty(\R^n)}  \leq C(1 + |p|),
\end{equation}
where $C>0$ depends only on $n$, $\alpha$, $\alpha'$, $\beta$ and $\beta'$. Note that the uniform Lipschitz estimate \eqref{eq:wd-osc} has nothing to do with the periodicity of $w^\delta$. In fact, one may use the weak Bernstein method \cite{B} to verify this uniform regularity, due to the structure conditions \eqref{eq:A-ellip}, \eqref{eq:A-C01}, \eqref{eq:H-convex} and \eqref{eq:H-Ck-C01}.

Now that $w^\delta$ is periodic, we may deduce from the interior gradient estimate for viscous Hamilton-Jacobi equations that Hence, by periodicity, $\osc_{\R^n} w^\delta \leq C(1+|p|)$, which yields that $w^\delta - w^\delta (0) \in C^{0,1}(\R^n)$ and 
\begin{equation}\label{eq:wtd-C01}
\norm{w^\delta - w^\delta(0)}_{L^\infty(\R^n)} + \norm{Dw^\delta}_{L^\infty(\R^n)} \leq C(1+|p|),
\end{equation}
where $C>0$ depend only on $n$, $\alpha$, $\alpha'$, $\beta$, $\beta'$ and $K$. Here $K$ is the constant appearing in the regularity assumption \eqref{eq:H-Ck-C01}. 

Due to \eqref{eq:H-quad}, \eqref{eq:dwd-Linf} and \eqref{eq:wd-osc}, we know that 
\begin{equation*}
\norm{ H(D w^\delta + p,\cdot) + \delta w^\delta }_{L^\infty(\R^n)} \leq C(1+ |p|^2),
\end{equation*}
where $C>0$ depends only on $n$, $\alpha$, $\alpha'$, $\beta$, $\beta'$ and $K$. Considering the second and the third terms on the left hand side of \eqref{eq:wd-pde} as an external force, we may apply the interior $C^{2,\mu}$ estimates and use the periodicity of $w^\delta$ to derive that $w^\delta - w^\delta(0) \in C^{1,\mu}(\R^n)$ and 
\begin{equation}\label{eq:wtd-C1a}
(1 + |p|) \left( \norm{w^\delta - w^\delta(0)}_{L^\infty(\R^n)} + \norm{Dw^\delta}_{L^\infty(\R^n)} \right)+ \left[D w^\delta\right]_{C^\mu(\R^n)} \leq C(1+|p|^2).
\end{equation}

Now the $C^{1,\mu}$ regularity of $w^\delta$ yields that 
\begin{equation*}
\norm{ H(D w^\delta + p,\cdot) + \delta w^\delta }_{C^\mu (\R^n)} \leq C(1+ |p|^2).
\end{equation*}
Hence, it follows from the interior $C^{2,\mu}$ estimates and the periodicity of $w^\delta$ that $w^\delta - w^\delta(0) \in C^{2,\mu}(\R^n)$ and 
\begin{equation}\label{eq:wtd-C1a}
(1 + |p|) \left( \norm{w^\delta - w^\delta(0)}_{L^\infty(\R^n)} + \norm{Dw^\delta}_{L^\infty(\R^n)} \right)+ \norm{D^2 w^\delta}_{C^\mu(\R^n)} \leq C(1+|p|^2).
\end{equation}

Due to the compactness of both of the sequences $\{w^\delta - w^\delta(0)\}_{\delta>0}$ and $\{-\delta w^\delta \}_{\delta>0}$ in $C^{2,\mu}(\R^n)$, we know that $w^\delta - w^\delta(0) \ra w$ and $-\delta w^\delta \ra \gamma$  in $C^{2,\mu'} (\R^n)$, for any $0<\mu'<\mu$, for some $w\in C^{2,\mu}(\R^n)$ and some $\gamma\in\R$, along a subsequence. Now that viscosity solutions are stable under the uniform convergence, we know that $w$ is a viscosity solution of \eqref{eq:cell} with the limit $\gamma$ on the right hand side. This proves the existence part of Lemma \ref{lemma:cell}.

To investigate the uniqueness of $\gamma$, we suppose towards a contradiction that there is another real number $\gamma'$, corresponding to the same $p$, such that \eqref{eq:cell} has a periodic viscosity solution, say $w'$. Without losing any generality, let us assume $\gamma>\gamma'$. Then it is easy to see that $w'$ is a strict subsolution of \eqref{eq:cell}. However, due to the periodicity of $w'-w$, $w'-w$ attains a local maximum at some point, whence we arrive at a contradiction. Thus, $\gamma$ must be unique. 

The inequality \eqref{eq:cell-Linf} follows immediately from the inequality \eqref{eq:dwd-Linf} and the fact that $-\delta w^\delta \ra \gamma$ uniformly in $\R^n$. To see that the estimate \eqref{eq:cell-C2a} holds, we first observe from the convergence of $w^\delta - w^\delta(0) \ra w$ in $C^{2,\mu'}(\R^n)$, for any $0<\mu'<\mu$, and the estimate \eqref{eq:wtd-C1a} that $w\in C^{2,\mu}(\R^n)$ and satisfies \eqref{eq:cell-C2a}. Note that we used $w(0) = 0$, which follows from the construction of $w$. Now if $w'$ is another periodic viscosity solution of \eqref{eq:w-pde}, then due to the uniqueness that we have shown in the beginning of this proof, we have $w' - w'(0) = w$. Therefore, $w'$ satisfies \eqref{eq:cell-C2a}, which completes the proof of this lemma.
\end{proof}

Due to the uniqueness of $\gamma$ in Lemma \ref{lemma:cell}, we may define a functional $\bar{H}:\R^n\ra\R$ in such a way that for each $p\in\R^n$, $\bar{H}(p)$ is the unique real number for which the following PDE,
\begin{equation}\label{eq:w-pde}
- \tr(A(y) D^2 w) + H( Dw + p, y ) = \bar{H}(p)\quad\text{in }\R^n,
\end{equation}
has a periodic solution in $C^{2,\mu}(\R^n)$ (for any $0<\mu<1$). Moreover, the second part of Lemma \ref{lemma:cell} yields a functional $w:\R^n\times\R^n\ra\R$ such that for each $p\in\R^n$, $w(p,\cdot)\in C^{2,\mu}(\R^n)$ (for any $0<\mu<1$) is the unique periodic viscosity solution of \eqref{eq:w-pde} that is normalized so as to satisfy
\begin{equation}\label{eq:w-0}
w(p,0) = 0.
\end{equation}

Let us list up some basic properties of $\bar{H}$ that were already found in \cite{E}. We provide the proof for the sake of completeness. 

\begin{lemma}\label{lemma:Hb} $\bar{H}$ satisfies the following properties.
\begin{enumerate}[(i)]
\item $\bar{H}$ has the same quadratic growth as that of $H$: 
\begin{equation}\label{eq:Hb-quad}
\alpha|p|^2 - \alpha'  \leq \bar{H}(p) \leq \beta|p|^2 + \beta',
\end{equation} 
for any $p\in\R^n$.
\item $\bar{H}$ is also convex:
\begin{equation}\label{eq:Hb-convex}
\bar{H} (tp + (1-t)q) \leq t \bar{H}(p) + (1-t) \bar{H}(q),
\end{equation}
for any $0\leq t\leq 1$, and any $p,q\in\R^n$. 
\item $\bar{H} \in C_{loc}^{0,1}(\R^n)$ and 
\begin{equation}\label{eq:Hb-C01}
| \bar{H}(p) - \bar{H}(q) | \leq C(1+|p| + |q|)|p-q|,
\end{equation}
where $C>0$ depends only on $n$, $\lambda$, $\Lambda$, $\alpha$, $\alpha'$, $\beta$, $\beta'$ and $K$. 
\end{enumerate}
\end{lemma} 

\begin{proof} Notice that \eqref{eq:Hb-quad} follows immediately from \eqref{eq:cell-Linf}. Thus, we only prove ($\bar{H2}$) and \eqref{eq:Hb-C01}. 

For the notational convenience, let us write $w_p(y)= w(p,y)$. To prove \eqref{eq:Hb-convex}, we assume to the contrary that there are some $p,q\in\R^n$ and $0<t<1$ such that \begin{equation}\label{eq:Hb-false}
t\bar{H}(p) + (1-t )\bar{H}(q) < \bar{H}(tp + (1-t)q).
\end{equation}
For the notational convenience, let us write $r = tp + (1-t)q$ and $\tilde{w}_r = tw_p + (1-t)w_q$. Then due to \eqref{eq:Hb-false} and \eqref{eq:H-convex}, one can easily deduce that $\tilde{w}_r$ is a periodic viscosity solution of 
\begin{equation}\label{eq:wtr-pde}
- \tr(A(y) D^2 \tilde{w}_r) + H (D \tilde{w}_r + r,y) < \bar{H}(r)\quad\text{in }\R^n.
\end{equation}
In other words, $\tilde{w}_r$ is a strict viscosity subsolution of the PDE that $w_r$, which is precisely \eqref{eq:cell} with $p$ replaced by $r$. Therefore, it follows from the comparison principle that $\tilde{w}_r - w_r$ cannot attain any local maximum. However, as $\tilde{w}_r - w_r$ being a non-constant continuous periodic function, it surely attains local maximum at some point, whence we arrive at a contradiction. Therefore, we must have \eqref{eq:Hb-convex} for any $0\leq t\leq 1$ and any $p,q\in\R^n$. 

Finally let us prove \eqref{eq:Hb-C01}. To do so, we go back to the penalized problem \eqref{eq:wd-pde}. Analogous with the notation $w_p$, let us denote by $w_p^\delta$ the unique viscosity solution of \eqref{eq:wd-pde} corresponding to $p$. Due to the uniform gradient estimate \eqref{eq:wd-osc} and the regularity assumption \eqref{eq:H-Ck-C01}, we have
\begin{equation*}
| H( Dw_p^\delta +p,y) - H(D w_q^\delta + q,y) | \leq C( 1+ |p| + |q| ) |p-q|,
\end{equation*}
where $C>0$ depends only on $n$, $\lambda$, $\Lambda$, $\alpha$, $\alpha'$, $\beta$, $\beta'$ and $K$. Therefore, we have
\begin{equation*}
\begin{split}
 - \tr (A(y) D^2 w_p^\delta) + H(D w_p^\delta + q, y) + \delta w_p^\delta \leq C(1+|p| + |q|) |p-q|\quad\text{in }\R^n,
\end{split}
\end{equation*}
in the viscosity sense. In other words, $w_p^\delta - \delta^{-1} C ( 1+ |p| + |q|)|p-q|$ is a viscosity subsolution of \eqref{eq:wd-pde} with $p$ replaced by $q$. Hence, it follows from the comparison principle that 
\begin{equation*}
\delta w_p^\delta - \delta w_q^\delta \leq C (1 + |p| + |q|)|p-q|,
\end{equation*}
on $\R^n$. Passing to the limit $\delta \ra 0$ in the last inequality, we arrive at
\begin{equation*}
\bar{H}(p) - \bar{H}(q) \leq C(1+ |p| + |q|) | p - q|
\end{equation*}
Similarly, one may also obtain that
\begin{equation*}
\bar{H}(q) - \bar{H}(p) \leq C(1+|p| + |q|) | p - q|,
\end{equation*}
proving \eqref{eq:Hb-C01}. This completes the proof of Lemma \ref{lemma:Hb}.
\end{proof}

%%%%%%%%%%%%%%%%%%%%%%%%%%%%%%%%%%%%%%%%%%%%%%%%%%%%%%%%%%%%%%%%%%%%%%%%%%%%%%%%%

\section{Regularity in the Slow Variable}\label{section:reg}

In this section, we shall investigate the regularity of $\bar{H}$ and $w$ in the slow variable $p$. Such a regularity has been established in the authors' previous works \cite{KL1} and \cite{KL2}, for fully nonlinear elliptic and, respectively, parabolic PDEs. Let us first observe the continuity of $w$ in $p$ variable. 

\begin{lemma}\label{lemma:w-Linf-C2a} $w\in C(\R^n;C^{2,\mu}(\R^n))$, for any $0<\mu<1$, and 
\begin{equation}\label{eq:w-Linf-C2a}
(1+|p|) \left( \norm{w (p,\cdot)}_{L^\infty(\R^n)} \norm{ D_y w(p,\cdot)}_{L^\infty(\R^n)} \right) + \norm{ D_y^2 w(p,\cdot) }_{C^\mu(\R^n)} \leq C(1+|p|^2),
\end{equation}
where $C>0$ depends only on $n$, $\lambda$, $\Lambda$, $\alpha$, $\alpha'$, $\beta$, $\beta'$, $K$ and $\mu$. 
\end{lemma}

\begin{proof} Let us fix $0<\mu<1$. The estimate \eqref{eq:w-Linf-C2a} follows immediately from \eqref{eq:cell-C2a} and the choice of $w$ that $w(p,0) = 0$. Thus, we prove that $w$ is continuous in $p$ variable with respect to the $C^{2,\mu}$ norm in $y$ variable. 

Let $\{p_k\}_{k=1}^\infty$ be a sequence of vectors in $\R^n$ converging to some $p_0\in\R^n$ as $k\ra\infty$. Let us write, for the notational convenience, $w_k (y) = w(p_k,y)$ and $\gamma_k = \bar{H}(p_k)$ for $k=0,1,2,\cdots$. We already know from \eqref{eq:Hb-C01} that $\gamma_k\ra \gamma_0$ as $k\ra\infty$. Hence, it suffices to prove that $w_k \ra w_0$ in $C^{2,\mu'}(\R^n)$ as $k\ra\infty$, for any $0<\mu'<\mu$.  

Due to \eqref{eq:w-Linf-C2a}, we know that $\{w_k\}_{k=1}^\infty$ is uniformly bounded in $C^{2,\mu}(\R^n)$, for any $0<\mu<1$. Now that $w_k$ is periodic for all $k=1,2,\cdot,s$, the Arzela-Ascoli theorem yields that for any subsequence $\{v_k\}_{k=1}^\infty \subset \{w_k\}_{k=1}^\infty$ there are a further subsequence $\{v_{k_i}\}_{i=1}^\infty$ and a periodic function $v\in C^{2,\mu}(\R^n)$ such that $v_{k_i}\ra v$ in $C^{2,\mu}(\R^n)$, for any $0<\mu<1$, as $i\ra\infty$. Now that $p_{k_i}\ra p_0$ and $\gamma_{k_i}\ra \gamma_0$ as $i\ra\infty$, we deduce from the stability of viscosity solutions that $w'$ and $\gamma'$ satisfies
\begin{equation*}
- \tr(A(y) D^2 v) + H(Dv + p_0,y) = \gamma_0 \quad \text{in }\R^n. 
\end{equation*}
Since $v(0) = 0$, the second part of Lemma \ref{lemma:cell} implies that $v = w_0$. This shows that any subsequence of $\{w_k\}_{k=1}^\infty$ contains a further subsequence that converges to $w_0$ in $C^{2,\mu'}(\R^n)$, for any $0<\mu'<\mu$. Therefore, $w_k \ra w_0$ in $C^{2,\mu'}(\R^n)$, for any $0<\mu'<\mu$ as $k\ra\infty$, which completes the proof. 
\end{proof}

Next we prove that $\bar{H}$ and $w$ are continuously differentiable in $p$. 

\begin{lemma}\label{lemma:w-C1-C2a} $\bar{H} \in C^1(\R^n)$ and 
\begin{equation}\label{eq:Hb-C1}
| D_p \bar{H}(p)| \leq C(1+|p|),
\end{equation}
where $C>0$ depends only on $n$, $\lambda$, $\Lambda$, $\alpha$, $\alpha'$, $\beta$, $\beta'$ and $K$. Moreover, $w\in C^1(\R^n; C^{2,\mu}(\R^n))$, for any $0<\mu<1$, such that for any $L>0$ and any $p\in B_L$, 
\begin{equation}\label{eq:w-C1-C2a}
\norm{ D_p w(p,\cdot)}_{C^{2,\mu}(\R^n)} \leq C_L,
\end{equation} 
where $C_L>0$ depends only on $n$, $\lambda$, $\Lambda$, $\alpha$, $\alpha'$, $\beta$, $\beta'$, $K$, $\mu$ and $L$.  
\end{lemma}

\begin{proof} Let us fix $0<\mu<1$. Throughout this proof, we shall write by $C_{*,\cdots,*}$ a positive constant depending at most on the parameters on the subscripts, if any, as well as $n$, $\lambda$, $\Lambda$, $\alpha$, $\alpha'$, $\beta$, $\beta'$, $K$ and $\mu$. We will also let it differ from one line to another, unless stated otherwise. 

Fix $L>0$, $p\in B_L$, $0<\mu<1$ and $1\leq k\leq n$. Write $w_h (y) = w(p+h e_k,y)$ and $\gamma_h = \bar{H}(p+he_k)$ for any $h\in\R$ with $|h|\leq 1$. Also write $W_h (y) = h^{-1} (w_h(y) - w_0(y))$, and $\Gamma_h = h^{-1}(\gamma_h - \gamma_0)$. Then $W_h$ turns out to be a periodic viscosity solution to
\begin{equation}\label{eq:Wh-pde}
- \tr(A(y) D^2 W_h) + B_h(y)\cdot (D W_h + e_k) = \Gamma_h\quad\text{in }\R^n,
\end{equation}
where 
\begin{equation*}
B_h(y) = \int_0^1 D_p H( t D_y w_h + (1-t) D_y w_0 + p + the, y)dt.
\end{equation*}

It follows from \eqref{eq:w-Linf-C2a} and \eqref{eq:H-Ck-C01} that $B_h \in C^\mu(\R^n)$ and
\begin{equation}\label{eq:Bh-Ca}
\norm{ B_h }_{C^\mu(\R^n)} \leq C( 1+ |p|),
\end{equation}
for any $h\in\R$ with $|h|\leq 1$. Moreover, we know from \eqref{eq:Hb-C01} that 
\begin{equation*}
| \Gamma_h | \leq C(1 + |p|),
\end{equation*}
for any $h\in\R$ with $0<|h|\leq 1$. Let us point out that the constant $C$ in the last inequality for $\Gamma_h$ depends only on $n$, $\alpha$, $\alpha'$, $\beta$, $\beta'$ and $K$. 

One may notice that \eqref{eq:Wh-pde} belongs to the same class of \eqref{eq:cell-lin}, whence it follows from Lemma \ref{lemma:cell-lin} below that $W_h\in C^{2,\mu}(\R^n)$ and 
\begin{equation}\label{eq:Wh-C2a}
\norm{ W_h }_{C^{2,\mu}(\R^n)} \leq C_L,
\end{equation}
for any $h\in\R$ with $0<|h|\leq 1$. On the other hand, from the fact that Lemma \ref{lemma:w-Linf-C2a} implies $D w_h \ra D w_0$ in $C^{1,\mu}(\R^n)$, we know that $B_h \ra B_0$ in $C^\mu(\R^n)$, where $B_0$ is defined by
\begin{equation*}
B_0(y) = D_p H (D_y w_0 + p, y).
\end{equation*}
As with the estimate \eqref{eq:Bh-Ca}, we also know that
\begin{equation}\label{eq:B0-Ca}
\norm{ B_0 }_{C^\mu(\R^n)} \leq C(1+|p|). 
\end{equation}

According to the Arzela-Ascoli theorem, there is some $W_0\in C^{2,\mu}(\R^n)$ such that $W_h \ra W_0$ in $C^{1,\mu'}(\R^n)$ for any $0<\mu'<\mu$, along a subsequence. Moreover, we may choose $\Gamma_0\in\R$ such that $\Gamma_h \ra \Gamma_0$ along a further subsequence. Then by the stability of viscosity solutions, $W_0$ becomes a periodic solution to
\begin{equation}\label{eq:W0-pde}
- \tr (A(y) D^2 W_0) + B_0 (y) \cdot (D_y W_0 + e_k) = \Gamma_0\quad\text{in }\R^n.
\end{equation}

Now that \eqref{eq:W0-pde} belongs to the same class of \eqref{eq:cell-lin}, it follows from Lemma \ref{lemma:cell-lin} below that $\Gamma_0$ is unique, and satisfies
\begin{equation}\label{eq:Gamma-Linf}
|\Gamma_0| \leq C(1+|p|),
\end{equation}
due to \eqref{eq:B0-Ca}. From the uniqueness of the limit $\Gamma_0$, we infer that $\Gamma_h \ra \Gamma_0$ without extracting any subsequence. By definition, $\Gamma_0 = D_{p_k} \bar{H}(p)$. 

Moreover, since any limit $W_0$ of $\{W_h\}_{0<|h|\leq 1}$ satisfies $W_0(0) = 0$, we also have from the last part of Lemma \ref{lemma:cell-lin} below that $W_0$ is unique, and belongs to $C^{2,\mu}(\R^n)$, with the estimate
\begin{equation}\label{eq:W0-C2a}
\norm{W_0}_{C^{2,\mu}(\R^n)} \leq C_L.
\end{equation}
Owing to the uniqueness of the limit $W_0$, again we conclude that $W_h\ra W_0$ in $C^{2,\mu}(\R^n)$ along the full sequence, which implies that $W_0 = D_{p_k} w(p,\cdot)$.

The continuity of $D_{p_k} \bar{H}$ and $D_{p_k} w$ in variable $p$ can be proved similarly as in the proof of Lemma \ref{lemma:w-Linf-C2a}. To avoid repeating arguments, we omit the details and leave this part to the reader. 
\end{proof}

\begin{lemma}\label{lemma:cell-lin} Let $B\in C^\mu(\R^n)$ be a periodic, vector-valued mapping. Then for each $p\in\R^n$, there exists a unique real number, $\gamma$, for which the following PDE,
\begin{equation}\label{eq:cell-lin}
-\tr (A(y) D^2 v) + B(y)\cdot (D v + p) = \gamma\quad\text{in }\R^n,
\end{equation}
admits a periodic viscosity solution $v\in C^{2,\mu}(\R^n)$. Moreover, $\gamma$ satisfies
\begin{equation}\label{eq:cell-lin-Linf}
|\gamma| \leq |p| \norm{B}_{L^\infty(\R^n)}.
\end{equation}
Furthermore, a periodic viscosity solution $v$ of \eqref{eq:cell-lin} is unique up to an additive constant, and satisfies
\begin{equation}\label{eq:cell-lin-C2a}
\norm{ v - v(0) }_{C^{2,\mu}(\R^n)} \leq C |p|,
\end{equation}
where $C>0$ depends only on $n$, $\lambda$, $\Lambda$, $\mu$ and $\norm{B}_{C^\mu(\R^n)}$. 
\end{lemma} 

\begin{proof} The proof is essentially the same with that of Lemma \ref{lemma:cell}, and hence it is omitted. 
\end{proof}

In what follows, let us write $\bar{B}(p) = D_p \bar{H}(p)$, $v(p,y) = D_pw(p,y)$ and 
\begin{equation}\label{eq:b}
B(p,y) = D_p H ( D_y w(p,y) + p, y).
\end{equation}
In view of the proof of Lemma \ref{lemma:w-C1-C2a}, we may understand $\bar{B}(p)$ as the unique real vector in $\R^n$ for which the following (decoupled) system,
\begin{equation}\label{eq:v-pde}
-\tr(A(y) D^2 v) + B(p,y)\cdot ( D_y v + I )  = \bar{B}(p),
\end{equation}
has a periodic viscosity solution, where $I$ is the identity matrix in $\sS^n$. Moreover, $v(p,\cdot)$ can be considered as the unique periodic viscosity solution of \eqref{eq:v-pde} such that
\begin{equation}\label{eq:v-0}
v(p,0) = 0.
\end{equation}

It is remarkable that after linearization in \eqref{eq:cell}, we end up with a cell problem whose gradient part has a linear growth, as shown in \eqref{eq:v-pde}. Moreover, one may expect that the linear structure of the ``new'' cell problem \eqref{eq:v-pde} will be preserved throughout the linearization we do in the future to obtain higher regularity of $\bar{H}$ and $w$ in $p$. This is the brief idea behind the proof for the following proposition. One may find a similar proposition for uniformly elliptic, fully nonlinear PDEs in the authors' previous work \cite{KL1} and \cite{KL2}. 

\begin{lemma}\label{lemma:w-Ck-C2a} $\bar{H}\in C^\infty(\R^n)$ and $w\in C^\infty(\R^n;C^{2,\mu}(\R^n))$, for any $0<\mu<1$, such that for any $k=0,1,2,\cdots$, any $L>0$ and any $p\in B_L$,
\begin{equation}\label{eq:w-Ck-C2a}
\left| D_p^k \bar{H}(p) \right| + \norm{ D_p^k w(p,\cdot) }_{C^{2,\mu}(\R^n)} \leq C_{k,L},
\end{equation}
where $C_{k,L}>0$ depends only on $n$, $\alpha$, $\alpha'$, $\beta$, $\beta'$, $K$, $\mu$, $k$ and $L$. 
\end{lemma}

\begin{proof} We follow the proof of Lemma \ref{lemma:w-C1-C2a}. Due to Lemma \ref{lemma:w-C1-C2a} and the regularity assumption \eqref{eq:H-Ck-C01}, we already know that $B\in C^1(\R^n;C^{2,\mu}(\R^n))$, for any $0<\mu<1$, with $B$ defined in \eqref{eq:b}. Thus, in order to run the same argument in the proof of Lemma \ref{lemma:w-C1-C2a}, we need the Lipschitz regularity of $\bar{B} = D_p\bar{H}$ in $p$. However, this can be shown as in the proof for \eqref{eq:Hb-C01} of Lemma \ref{lemma:Hb}. This is because we can also understand the constant vector $\bar{B}(p)$ as the limit of $\{-\delta v^\delta\}_{\delta>0}$, with $v^\delta$ being the unique periodic viscosity solution of 
\begin{equation}\label{eq:vd-pde}
-\tr(A(y) D^2 v^\delta) + B(p,y) \cdot (D_y v^\delta + I) + \delta v^\delta = 0\quad\text{in }\R^n.
\end{equation}

Once we know that $\bar{B}$ is Lipschitz in $p$, then it follows from Lemma \ref{lemma:w-C1-C2a} and the elliptic regularity theory that the difference quotient $V_h = h^{-1}( v_h -  v_0)$, being a periodic viscosity solution of 
\begin{equation}\label{eq:Vh-pde}
-\tr(A(y) D^2 V_h) + B_h(y) \cdot D_y V_h + B_h(y) \cdot (D_y v_0(y) + I) = \bar{B}_h\quad\text{in }\R^n,
\end{equation}
with $v_h = v(p+he_k,\cdot)$, $B_h = B(p+he_k,\cdot)$, $B_h = h^{-1} (B_h - B_0)$ and $\bar{B}_h = h^{-1}(\bar{B}_h - \bar{B}_0)$, is uniformly bounded in $C^{2,\mu}(\R^n)$. Hence, we deduce from the stability of viscosity solutions that any pair $(V_0,\bar{B}_0)$ of $\{V_h\}_{0<|h|\leq 1}$ and, respectively, $\{\bar{B}_h\}_{0<|h|\leq 1}$ must satisfy 
\begin{equation}\label{eq:V0-pde}
-\tr(A(y) D^2 V_0) + B_0(y) \cdot D_y V_0 + B_0(y) \cdot (D_y v_0(y) + I) = \bar{B}_0\quad\text{in }\R^n.
\end{equation}

Since \eqref{eq:V0-pde} belongs to the same class of \eqref{eq:cell-lin}, we know from Lemma \ref{lemma:cell-lin} that $V_0$ and $\bar{B}_0$ are unique. Thus, we derive the differentiability of $\bar{B}$ and $v$ in $p$. Arguing as in the proof of Lemma \ref{lemma:w-Linf-C2a}, we may also observe that $D_p \bar{B}$ and $D_p v$ are continuous in $p$.

One may now iterate this argument to obtain higher regularity of $\bar{B}$ and $v$ in $p$, which automatically implies that of $\bar{H}$ and $w$. We leave out the details to the reader. 
\end{proof}

%%%%%%%%%%%%%%%%%%%%%%%%%%%%%%%%%%%%%%%%%%%%%%%%%%%%%%%%%%%%%%%%%%%%%%%%%%%%%%%%%

\section{Interior Corrector and Higher Order Convergence Rate}\label{section:cor}

In this section, we construct the higher order interior correctors for the homogenization problem \eqref{eq:HJ}, based on the regularity result achieved in Section \ref{section:reg}. 

We begin with the effective Hamilton-Jacobi equation for \eqref{eq:HJ}, which is given by
\begin{equation}\label{eq:eff-HJ}
\begin{dcases}
\p_t \bar{u}_0 + \bar{H} ( D\bar{u}_0 ) = 0 & \text{in }\R^n\times(0,\infty),\\
\bar{u}_0 = g & \text{on }\R^n\times\{t=0\}.
\end{dcases}
\end{equation}
The characteristic curve, which starts from $x\in\R^n$, is given by 
\begin{equation}\label{eq:char}
\xi(t;x) = x + D_p \bar{H} ( D_x g(x)) t.
\end{equation}
Note that this is indeed a line with direction $D_p \bar{H} ( D_x g(x))$. Moreover, the gradient of $\bar{u}$ is constant along this curve. To be specific, we have 
\begin{equation}\label{eq:char-Dub}
D_x\bar{u}(\xi(t;x),t) = D_x g(x).
\end{equation} 

It is noteworthy that the initial data, $g$, does not play any role when deriving the effective Hamiltonian $\bar{H}$, as shown in Section \ref{section:reg}. This allows us to choose the initial data $g$ {\it a posteriori} so as to make sure that 
\begin{equation}\label{eq:char-inj}
\{\xi(t;x) : t>0\} \bigcap \{ \xi(t;x') : t>0\} = \emptyset,
\end{equation}
if and only if $x\neq x'$, as well as that  
\begin{equation}\label{eq:char-surj}
\bigcup_{x\in\R^n} \{\xi(t;x) : t>0\} = \R^n.
\end{equation}

One may easily observe that there are infinitely many initial data $g$ that satisfy the conditions \eqref{eq:char-inj} and \eqref{eq:char-surj}, once $\bar{H}$ is determined. A trivial example is any affine function whose gradient is a non-vanishing point of $D_p \bar{H}$. Note that the non-vanishing set of $D_p \bar{H}$ is always open and non-empty, since $\bar{H}$ is convex and grows quadratically at the infinity. A rather non-trivial example is any smooth, convex and globally Lipschitz function whose gradients are contained in the non-vanishing set of $D_p \bar{H}$.

Once we have the initial data $g$, we know from the characteristic equations for \eqref{eq:eff-HJ} that $\bar{u}_0\in C^\infty(\R^n\times[0,\infty))$. Setting 
\begin{equation}\label{eq:Bb}
\bar{B}(x,t) = D_p \bar{H}(D_x \bar{u}_0(x,t)),
\end{equation}
we obtain $\bar{B} \in C^\infty(\R^n\times[0,\infty))$, according to Lemma \ref{lemma:w-Ck-C2a}.

In order to have a regular solution for the first order linear PDE whose drift term is associated with $\bar{B}$, we require that 
\begin{equation}\label{eq:Bb-0}
\bar{B} (x,t) \neq 0,
\end{equation}
for any $(x,t)\in\R^n\times(0,\infty)$. Since the image of $\bar{B}$ on $\R^n\times(0,\infty)$ coincides with that of $D_p \bar{H} ( D_x g)$ on $\R^n$, we ask $D_x g$ not to be the critical points of $D_p \bar{H}$. 

Let us list up the conditions for $g$ to be imposed in the rest of this paper:
\begin{enumerate}[(i)]
\item $g$ is convex:
\begin{equation}\label{eq:g-convex}
g(tx+ (1-t)x') \leq t g(x) + (1-t) g(x'),
\end{equation}
for any $0\leq t\leq 1$ and any $x,x'\in\R^n$. 
\item $g \in C^\infty(\R^n)\cap \Lip(\R^n)$, and there is $L>0$ such that
\begin{equation}\label{eq:g-Ck}
\norm{D_x^k g}_{L^\infty(\R^n)} \leq L,
\end{equation}
for any $k=1,2,\cdots$. Moroever, $g$ is normalized such that 
\begin{equation}\label{eq:g-0}
g(0) = 0.
\end{equation}
\item $D_x g$ is not the critical points of $D_p \bar{H}$: 
\begin{equation}\label{eq:DHDg-0}
D_p \bar{H} (D_x g(x)) \neq 0,
\end{equation}
for any $x\in\R^n$. 
\end{enumerate}

Under the assumptions \eqref{eq:g-convex} -- \eqref{eq:DHDg-0} on $g$, altogether with the properties \eqref{eq:Hb-quad} -- \eqref{eq:Hb-C01} and \eqref{eq:w-Ck-C2a} of $\bar{H}$, we know from the standard regularity theory for Hamilton-Jacobi equations that $\bar{u}_0 \in C^\infty(\R^n\times[0,\infty))$, and in particular, we have, for each $i,j=0,1,2,\cdots$, and any $T>0$,
\begin{equation}\label{eq:ub0-Ck}
\left| D_x^i\p_t^j \bar{u}_0(x,t) \right| \leq C_{i,j,T},
\end{equation}
uniformly for all $(x,t)\in\R^n\times[0,T]$, where $C_{i,j,T}$ is a positive constant depending at most on $n$, $\alpha$, $\alpha'$, $\beta$, $\beta'$, $K$, $L$, $i$, $j$ and $T$. 

Moreover, due to \eqref{eq:ub0-Ck} and \eqref{eq:w-Ck-C2a}, we know that $\bar{B} \in C^\infty(\R^n\times[0,\infty))$ and, for each $i,j=0,1,2,\cdots$, and any $T>0$,
\begin{equation}\label{eq:Bb-Ck}
\left| D_x^i \p_t^j \bar{B} (x,t) \right| \leq C_{i,j,T},
\end{equation}
uniformly for all $(x,t)\in \R^n \times[0,T]$, where $C_{i,j,T}$ is another positive constant determined by the same parameters listed above.

In what follows, we shall seek a sequence of the interior correctors for the homogenization problem \eqref{eq:HJ}. The first order interior corrector $w_1$ will be in the form of 
\begin{equation}\label{eq:w1}
w_1 (x,t,y) = \phi_1(x,t,y) + \bar{u}_1(x,t),
\end{equation}
where $\phi_1$ denotes 
\begin{equation}\label{eq:phi1}
\phi_1(x,t,y) = w(D_x \bar{u}_0(x,t),y),
\end{equation}
with $w = w(p,y)$ being the periodic (viscosity) solution of \eqref{eq:w-pde} normalized so as to satisfy \eqref{eq:w-0}. Here $\bar{u}_1$ is an effective data that is not determined yet. Let us remark that one may choose $\bar{u}_1$ by any regular data, if one stops seeking interior correctors at this step. However, if one would like to go further and construct the second order corrector $w_2$, one needs to select $\bar{u}_1$ specifically by the solution of an effective limit equation, which arises from the solvability condition of $w_2$. 

We will continuously observe such a relationship between the consecutive correctors. In fact, the $k$-th order interior corrector $w_k$, for $k\geq 2$, will be in the form of
\begin{equation}\label{eq:wk}
w_k (x,t,y) = \phi_k(x,t,y) + \chi(x,t,y)\cdot D_x\bar{u}_{k-1}(x,t) + \bar{u}_k(x,t),
\end{equation}
where $\phi_k(x,t,\cdot)$ will be the periodic viscosity solution of a certain cell problem normalized so as to satisfy $\phi_k(x,t,0) = 0$, and $\chi:\R^n\times[0,\infty)\times\R^n\ra\R^n$ will be defined by 
\begin{equation}\label{eq:chi}
\chi(x,t,y) = v( D_x\bar{u}_0(x,t),y),
\end{equation}
with $v=v(p,y)$ being the periodic solution of \eqref{eq:v-pde} normalized so as to satisfy \eqref{eq:v-0}. Here $\bar{u}_{k-1}$ will be determined specifically such that the cell problem for $\phi_k$ is solvable, while $\bar{u}_k$ will be ``free'' to choose before one tries to construct the $(k+1)$-th corrector $w_{k+1}$. 

It is noteworthy that, owing to Lemma \ref{lemma:w-Ck-C2a}, we have $\chi \in C^\infty(\R^n\times[0,\infty);C^{2,\mu}(\R^n))$ and, for any $i,j=0,1,2,\cdots$ and any $T>0$,
\begin{equation}\label{eq:chi-Ck-C2a}
\norm{ D_x^i \p_t^j \chi(x,t,\cdot)}_{C^{2,\mu}(\R^n)} \leq C_{i,j,R,T},
\end{equation}
uniformly for all $(x,t)\in \R^n \times[0,T]$. In addition, we know from \eqref{eq:v-pde} and \eqref{eq:v-0} that for each $(x,t)\in\R^n\times(0,\infty)$, $\chi(x,t,\cdot)$ is the unique periodic viscosity solution of 
\begin{equation}\label{eq:chi-pde}
-\tr(A(y) D_y^2 \chi) + B(x,t,y)\cdot( D_y\chi + I) = \bar{B}(x,t)\quad\text{in }\R^n,
\end{equation}
which also satisfies 
\begin{equation}\label{eq:chi-0}
\chi(x,t,0) = 0. 
\end{equation}

For the rest of this section, we will justify the existence of the higher order interior correctors in a rigorous way. The corresponding work has been done by the authors in \cite{KL1} and \cite{KL2} in the framework of fully nonlinear, uniformly elliptic, second order PDEs in non-divergence form. 

To simplify the notation, let us write 
\begin{equation}\label{eq:w0}
w_0(x,t,y) = \bar{u}_0(x,t),
\end{equation}
and by $W_k$, for $k=0,1,2,\cdots$, the vector-valued mapping,
\begin{equation}\label{eq:Wk}
W_k (x,t,y) = D_y w_{k+1} (x,t,y) + D_x w_k(x,t,y).
\end{equation}
Note from \eqref{eq:w1}, \eqref{eq:phi1} and \eqref{eq:w0} that
\begin{equation}\label{eq:W0}
W_0 (x,t,y) = D_y \phi_1 (x,t,y) + D_x \bar{u}_0(x,t). 
\end{equation}

We shall also write by $B_k$, for $k=1,2,\cdots$, the mapping, 
\begin{equation}\label{eq:Bk}
B_k(x,t,y) = D_p^k H(W_0(x,t,y)),
\end{equation}
where $D_p^k H$ is understood in the sense of Fr\'echet derivatives, and to make the notation coherent to the notation of $\bar{B}$, we will write
\begin{equation}\label{eq:B}
B (x,t,y) = B_1(x,t,y).
\end{equation}
Let us also remark that, due to \eqref{eq:H-Ck-C01}, \eqref{eq:ub0-Ck} and \eqref{eq:w-Ck-C2a}, we have  $B_k\in C^\infty(\R^n\times[0,\infty);C^\mu(\R^n))$, for any $0<\mu<1$. In particular, we obtain, for any $i,j=0,1,2,\cdots$, any $k=1,2,\cdots$ and any $T>0$,
\begin{equation}\label{eq:Bk-Ck-Ca}
\norm{ D_x^i \p_t^j B_k(x,t,\cdot)}_{C^\mu(\R^n)} \leq C_{i,j,k,T},
\end{equation}
uniformly for all $(x,t)\in \R^n \times[0,T]$, where $C_{i,j,k,T}>0$ depends only on $n$, $\lambda$, $\Lambda$, $\alpha$, $\alpha'$, $\beta$, $\beta'$, $K$, $L$, $\mu$, $i$, $j$, $k$ and $T$. 

\begin{lemma}\label{lemma:wk} Suppose that $A$, $H$ and $g$ satisfy \eqref{eq:A-peri} -- \eqref{eq:A-C01}, \eqref{eq:H-peri} -- \eqref{eq:H-Ck-C01} and, respectively, \eqref{eq:g-convex} -- \eqref{eq:DHDg-0}. Then there exists a sequence $\{w_k\}_{k=1}^\infty$ satisfying the following.
\begin{enumerate}[(i)]
\item $w_k \in C^\infty(\R^n\times[0,\infty); C^{2,\mu}(\R^n))$, for any $0<\mu<1$, and 
\begin{equation}\label{eq:wk-Ck-C2a}
\norm{ D_x^i \p_t^j w_k(x,t,\cdot)}_{C^{2,\mu}(\R^n)} \leq C_{i,j,k,T},
\end{equation}
for each $i,j=0,1,2,\cdots$, any $T>0$, and uniformly for all $(x,t)\in\R^n\times[0,T]$, where $C_{i,j,k,T}>0$ depends only on $n$, $\lambda$, $\Lambda$, $\alpha$, $\alpha'$, $\beta$, $\beta'$, $K$, $L$, $\mu$, $i$, $j$, $k$ and $T$. 
\item $w_k$ satisfies
\begin{equation}\label{eq:wk-0}
w_k(x,0,0) = 0.
\end{equation}
\item For each $(x,t)\in\R^n\times(0,\infty)$, $w_k(x,t,\cdot)$ is a periodic solution of
\begin{equation}\label{eq:w1-pde}
\p_t w_0(x,t,y) - \tr(A(y) D_y^2 w_1) + H( D_y w_1 + D_x w_0(x,t,y),y ) = 0 \quad\text{in }\R^n,
\end{equation}
for $k=1$, and 
\begin{equation}\label{eq:wk-pde}
\begin{split}
&\p_t w_{k-1}(x,t,y) - \tr(A(y) D_y^2 w_k) \\
&+ B(x,t,y) \cdot (D_y w_k + D_x w_{k-1}(x,t,y)) + \Phi_{k-1}(x,t,y) = 0 \quad\text{in }\R^n,  
\end{split}
\end{equation}
for $k\geq 2$, where 
\begin{equation}\label{eq:Phik}
\begin{split}
\Phi_{k-1} (x,t,y) &= -2\tr (A(y) (D_xD_y w_{k-1} (x,t,y) + D_x^2 w_{k-2} (x,t,y))) \\
&\quad + \sum_{l=2}^{k-1} \frac{1}{l!} \sum_{\substack{ i_1 + \cdots + i_l = k-1 \\ i_1,\cdots,i_1\geq 1}} B_l (x,t,y) ( W_{i_1}(x,t,y),\cdots, W_{i_l} (x,t,y)),
\end{split}
\end{equation}
with the last summation term understood as zero when $k=2$. 
\end{enumerate}
\end{lemma}

\begin{remark}\label{remark:wk} The summation term in the definition \eqref{eq:Phik} of $\Phi_k$ amounts to the nonlinear effect of the Hamiltonian $H$ in $p$. In view of \eqref{eq:Bk}, one may easily observe that the whole summation term becomes zero when $H$ is linear in $p$. The choice of $\Phi_k$ is specifically designed to achieve \eqref{eq:HDetame}, which will eventually leads us to the higher order convergence rate for the homogenization problem \eqref{eq:HJ}. We will also see later in \eqref{eq:Phik-nl} and \eqref{eq:HDetame-nl} that the choice of $\Phi_k$ changes according to the type of nonlinearity that needs to be taken care of.
\end{remark}

\begin{proof}[Proof of Lemma \ref{lemma:wk}] Throughout this proof, we shall fix $0<\mu<1$, and denote by $C_{*,\cdots,*}$ a positive constant depending only on the subscripts as well as the parameters $n$, $\lambda$, $\Lambda$, $\alpha$, $\alpha'$, $\beta$, $\beta'$, $K$, $L$ and $\mu$. We will also allow it to vary from one line to another, for notational convenience. 

Define $\phi_1$ by \eqref{eq:phi1}. Since $\bar{u}_0\in C^\infty(\R^n\times[0,\infty))$, we know from \eqref{eq:w-Ck-C2a} that $\phi_1 \in C^\infty(\R^n\times[0,\infty);C^{2,\mu}(\R^n))$. Moreover, it follows from \eqref{eq:ub0-Ck} that for each $i,j=0,1,2,\cdots$, and any $T>0$, 
\begin{equation}\label{eq:phi1-Ck-C2a}
\norm{ D_x^i\p_t^j \phi_1(x,t,\cdot) }_{C^{2,\mu}(\R^n)} \leq C_{i,j,T},  
\end{equation}
uniformly for all $(x,t)\in\R^n\times[0,T]$. In view of the definition of $w_0$ in \eqref{eq:w0}, $\phi_1(x,t,\cdot)$ is a periodic viscosity solution of \eqref{eq:w1-pde}, for each $(x,t)\in\R^n\times(0,\infty)$, as $\bar{u}_0$ being the solution of \eqref{eq:eff-HJ}.

Let us now fix $k\geq 1$ and suppose that we have already found $\{w_l\}_{l=0}^{k-1}$ that satisfies the assertions (i) and (ii) of this lemma. Moreover, assume that we have already obtained $\bar{u}_{k-1} \in C^\infty(\R^n\times[0,\infty))$ such that 
\begin{equation}\label{eq:ubk-1-Ck}
\left| D_x^i \p_t^j \bar{u}_{k-1} (x,t) \right| \leq C_{i,j,k,T},
\end{equation}
for any $i,j=0,1,2,\cdots$, any $T>0$ and any $(x,t)\in \R^n\times[0,T]$. Additionally, suppose that we have also found $\phi_k\in C^\infty(\R^n\times[0,\infty);C^{2,\mu}(\R^n))$ such that for each $(x,t)\in\R^n\times[0,\infty)$, $\phi_k (x,t,\cdot)$ is a periodic function normalized by 
\begin{equation}\label{eq:phik-0} 
\phi_k(x,t,0) = 0, 
\end{equation}
and that we have, for any $i,j=0,1,2,\cdots$ and any $T>0$, 
\begin{equation}\label{eq:phik-Ck-C2a}
\norm{ D_x^i \p_t^j \phi_k(x,t,\cdot)}_{C^{2,\mu}(\R^n)} \leq C_{i,j,k,T},
\end{equation}
uniformly for all $(x,t)\in \R^n \times[0,T]$. 

Define $\tilde{w}_k$ by 
\begin{equation}\label{eq:wt1}
\tilde{w}_1 (x,t,y) = \phi_1(x,t,y),
\end{equation}
if $k=1$, and by 
\begin{equation}\label{eq:wtk}
\tilde{w}_k (x,t,y) = \phi_k(x,t,y) + \chi(x,t,y) \cdot D_x \bar{u}_{k-1}(x,t),
\end{equation}
if $k\geq 2$. We deduce from \eqref{eq:ubk-1-Ck}, \eqref{eq:phik-Ck-C2a} and \eqref{eq:chi-Ck-C2a} that $\tilde{w}_k \in C^\infty(\R^n\times[0,\infty);C^{2,\mu}(\R^n))$ and satisfies 
\begin{equation}\label{eq:wtk-Ck-C2a}
\norm{ D_x^i \p_t^j \tilde{w}_k(x,t,\cdot)}_{C^{2,\mu}(\R^n)} \leq C_{i,j,k,T},
\end{equation}
for any $i,j=0,1,2,\cdots$, any $T>0$ and any $(x,t)\in \R^n \times[0,T]$. 

In view of the estimate \eqref{eq:wtk-Ck-C2a}, we observe that $\tilde{w}_k$ satisfies the assertion (i) of Lemma \ref{lemma:wk}. Moreover, it follows from the hypothesis \eqref{eq:phik-0}, and the fact \eqref{eq:chi-0} that $\tilde{w}_k$ verifies the assertion (ii) of this lemma as well. Henceforth, we shall assume, as the last hypothesis for this induction argument, that $\tilde{w}_k$ satisfies the assertion (iii) of this lemma.

In order to find $\bar{u}_k$, we first define
\begin{equation}\label{eq:fk}
\begin{split}
f_k (x,t,y) &= \p_t \tilde{w}_k (x,t,y) + B(x,t,y)\cdot D_x \tilde{w}_k(x,t,y)  \\
&\quad -2\tr ( A(y) (D_xD_y \tilde{w}_k (x,t,y) + D_x^2 w_{k-1} (x,t,y))) \\
&\quad + \sum_{l=2}^k \frac{1}{l!} \sum_{\substack{ i_1 + \cdots + i_l = k \\ i_1,\cdots,i_1\geq 1}} B_l (x,t,y) ( W_{i_1}(x,t,y),\cdots, W_{i_l} (x,t,y)). 
\end{split}
\end{equation}
Using \eqref{eq:A-C01}, \eqref{eq:Bk-Ck-Ca}, \eqref{eq:ubk-1-Ck}, \eqref{eq:phik-Ck-C2a}, \eqref{eq:chi-Ck-C2a} and \eqref{eq:wtk-Ck-C2a} together with the induction hypothesis \eqref{eq:wk-Ck-C2a}, we deduce that $f_k \in C^\infty(\R^n\times[0,\infty);C^\mu(\R^n))$ and  
\begin{equation}\label{eq:fk-Ck-Ca}
\norm{ D_x^i \p_t^j f_k(x,t,\cdot)}_{C^\mu(\R^n)} \leq C_{i,j,k,T},
\end{equation}
for any $i,j=0,1,2,\cdots$, any $T>0$ and any $(x,t)\in \R^n \times[0,T]$. 

Now that $f_k$ is periodic in $y$, we may consider the following cell problem: there exists a unique function $\bar{f}_k:\R^n\times(0,\infty)\ra\R$ such that for each $(x,t)\in\R^n\times[0,\infty)$, the PDE,
\begin{equation}\label{eq:phik-pde}
-\tr (A(y) D_y^2 \phi_{k+1}) + B(x,t,y) \cdot D_y \phi_{k+1} + f_k(x,t,y) = \bar{f}_k (x,t) \quad\text{in }\R^n,
\end{equation}
has a periodic viscosity solution. Following the argument in the proof of Lemma \ref{lemma:cell-lin}, we see that the cell problem \eqref{eq:phik-pde} is solvable. Moreover, if we normalize $\phi_{k+1}$ so as to satisfy 
\begin{equation}\label{eq:phik+1-0}
\phi_{k+1}(x,t,0) = 0,
\end{equation}
such a periodic viscosity solution $\phi_{k+1}$ is unique. Furthermore, applying the regularity theory in the slow variable established in Lemma \ref{lemma:w-Ck-C2a}, we deduce from \eqref{eq:Bk-Ck-Ca} and \eqref{eq:fk-Ck-Ca} that $\bar{f}_k \in C^\infty(\R^n\times(0,\infty))$ and $\phi_{k+1} \in C^\infty(\R^n\times[0,\infty);C^{2,\mu}(\R^n))$. In particular, we have, for any $i,j=0,1,2,\cdots$ and any $T>0$, 
\begin{equation}\label{eq:fbk-phik-Ck-C2a}
\left| D_x^i \p_t^j \bar{f}_k (x,t) \right| + \norm{ D_x^i \p_t^j \phi_{k+1}(x,t,\cdot)}_{C^{2,\mu}(\R^n)} \leq C_{i,j,k,T},
\end{equation}
uniformly for all $(x,t)\in \R^n \times[0,T]$.

With $\bar{f}_k$ at hand, we consider the first order linear PDE,
\begin{equation}\label{eq:ubk-pde}
\begin{cases}
\p_t \bar{u}_k + \bar{B}(x,t)\cdot D_x \bar{u}_k + \bar{f}_k(x,t) =0 & \text{in }\R^n\times(0,\infty),\\
\bar{u}_k = 0 & \text{on }\R^n\times\{t=0\},
\end{cases}
\end{equation}
where $\bar{B}$ is defined by \eqref{eq:Bb}. Recall from \eqref{eq:DHDg-0} that $\bar{B}$ vanishes nowhere in $\R^n\times(0,\infty)$. Thus, it follows from the classical existence theory for the first order linear PDE that there exists a unique solution $\bar{u}_k \in C^\infty(\R^n\times[0,\infty))$ of \eqref{eq:ubk-pde} such that 
\begin{equation}\label{eq:ubk-Ck}
\left| D_x^i \p_t^j \bar{u}_k (x,t) \right| \leq C_{i,j,k,T},
\end{equation}
for any $i,j=0,1,2,\cdots$, any $T>0$ and any $(x,t)\in \R^n\times[0,T]$. 

Define $w_k$ by 
\begin{equation}\label{eq:wk-re}
w_k(x,t,y) = \tilde{w}_k(x,t,y) + \bar{u}_k(x,t),
\end{equation}
which coincides with the expression \eqref{eq:w1} and \eqref{eq:wk} for any $k\geq 1$. Using \eqref{eq:wtk-Ck-C2a} and \eqref{eq:ubk-Ck}, we see that $w_k$, defined by \eqref{eq:wk-re}, verifies the assertions (i) and (ii) of Lemma \ref{lemma:wk}. Besides, let us notice that 
\begin{equation}\label{eq:fk-re}
f_k(x,t,y) = \p_t \tilde{w}_k(x,t,y) + B(x,t,y)\cdot D_x \tilde{w}_k(x,t,y) + \Phi_k(x,t,y),
\end{equation}
where $\Phi_k$ is defined by \eqref{eq:Phik}, since we have $D_y \tilde{w}_k (x,t,y) = D_y w_k(x,t,y)$. 

To this end, let us set $\tilde{w}_{k+1}$ by 
\begin{equation}\label{eq:wtk+1}
\tilde{w}_{k+1} (x,t,y) = \phi_{k+1}(x,t,y) + \chi(x,t,y) \cdot D_x \bar{u}_k(x,t).
\end{equation}
Then we observe from \eqref{eq:phik-pde}, \eqref{eq:ubk-pde}, \eqref{eq:chi-pde} and \eqref{eq:fk-re} that
\begin{equation}
\begin{split}
& \p_t w_k(x,t,y) - \tr( A(y) D_y^2 \tilde{w}_{k+1}(x,t,y)) \\
& + B(x,t,y)\cdot( D_y \tilde{w}_{k+1}(x,t,y) + D_x w_k(x,t,y) ) + \Phi_k(x,t,y) \\
& =  \p_t\bar{u}_k(x,t)  - \tr (A(y) D_y^2 \phi_{k+1}(x,t,y)) + B(x,t,y)\cdot \phi_{k+1} + f_k(x,t,y)  \\
&\quad + (- \tr( A(y) D_y^2 \chi(x,t,y)) + B(x,t,y) \cdot (D_y\chi(x,t,y) + I))\cdot D_x\bar{u}_k(x,t) \\
& =  \p_t\bar{u}_k(x,t)  +  \bar{f}_k(x,t) + \bar{B}(x,t)\cdot D_x\bar{u}_k(x,t) \\
& = 0.
\end{split}
\end{equation}
Hence, we have proved that $\tilde{w}_{k+1}$ satisfies the assertion (iii) of Lemma \ref{lemma:wk}. 

Recall that we have started with $\{w_l\}_{k=0}^{k-1}$, $\bar{u}_{k-1}$, $\phi_k$ and $\tilde{w}_k$, and obtained $w_k$, $\bar{u}_k$, $\phi_{k+1}$ and $\tilde{w}_{k+1}$ that satisfy all the induction hypotheses. Moreover, we have established the initial case for the induction hypotheses in the beginning of this proof. Thus, the proof is completed by the induction principle.
\end{proof}

We shall call $w_k$, chosen from Lemma \ref{lemma:wk}, the $k$-th order interior corrector for the homogenization problem \eqref{eq:HJ}, due to the following lemma. Although the computation involved in the proof below is similar to what can be found in \cite[Section 3.3]{KL1} and \cite[Section 4.1]{KL2}, we present it in detail for the sake of completeness. 

\begin{lemma}\label{lemma:wk-cor}
Let $\{w_k\}_{k=1}^\infty$ be chosen as in Lemma \ref{lemma:wk}. Then for each integer $m\geq 1$ and each $0<\e\leq\frac{1}{2}$, the function $\eta_m^\e$, defined by  
\begin{equation}\label{eq:etame}
\eta_m^\e (x,t) = \bar{u}_0(x,t) + \sum_{k=1}^m \e^k w_k\left(x,t,\frac{x}{\e}\right),
\end{equation}
is a viscosity solution of 
\begin{equation}\label{eq:etame-pde}
\begin{dcases}
\p_t \eta_m^\e -\e \tr \left( A\left(\frac{x}{\e}\right) D^2 \eta_m^\e\right)  + H\left( D \eta_m^\e,\frac{x}{\e}\right) = \psi_m^\e\left(x,t,\frac{x}{\e}\right) & \text{in }\R^n\times(0,\infty),\\
\eta_m^\e = g & \text{on }\R^n\times\{t=0\},
\end{dcases}
\end{equation}
where $\psi_m^\e \in C(\R^n\times[0,\infty);L^\infty(\R^n))$ satisfies, for any $T>0$, 
\begin{equation}\label{eq:psime-Linf}
\norm{\psi_m^\e(x,t,\cdot)}_{L^\infty(\R^n)} \leq C_{m,T}\e^m,
\end{equation}
uniformly for all $(x,t)\in \R^n\times[0,T]$, where $C_{m,T}>0$ is a constant depending only on $n$, $\lambda$, $\Lambda$, $\alpha$, $\alpha'$, $\beta$, $\beta'$, $K$, $L$, $\mu$, $m$ and $T$. 
\end{lemma}

\begin{proof} Aligned with the notation \eqref{eq:Wk} of $W_k$, let us denote by $X_k$, the matrix-valued mapping, 
\begin{equation}\label{eq:Xk}
X_k(x,t,y) = D_y^2 w_{k+1}(x,t,y) + (D_xD_y + D_yD_x) w_k(x,t,y) + D_x^2 w_{k-1}(x,t,y),
\end{equation}
for $k=1,2,\cdots$, with $w_{-1}$ being understood as the identically zero function. One may notice from \eqref{eq:phi1}, \eqref{eq:w1} and \eqref{eq:w0} that
\begin{equation}\label{eq:X0}
X_0(x,t,y) = D_y^2 \phi_1(x,t,y).
\end{equation}

Fix $m\geq 1$ and $0<\e\leq\frac{1}{2}$. For the moment, we shall replace $w_{m+1}$ and $w_{m+2}$ by the identically zero functions, only to simplify the exposition. With this replacement, we have $W_m = D_x w_m$, $X_m = (D_xD_y + D_yD_x) w_m + D_x^2 w_{m-1}$ and $X_{m+1} = D_x^2 w_m$.

In view of \eqref{eq:Wk} and \eqref{eq:Xk}, we have
\begin{equation*}
D \eta_m^\e (x,t) = \sum_{k=0}^m \e^k W_k\left(x,t,\frac{x}{\e}\right),
\end{equation*}
and
\begin{equation*}
\e D^2 \eta_m^\e (x,t) = \sum_{k=0}^{m+1} \e^k X_k\left(x,t,\frac{x}{\e}\right).
\end{equation*}
Let us define $\Psi_k$ by
\begin{equation}\label{eq:Psi0}
\Psi_0(x,t,y) = - \tr (A(y) X_0(x,t,y)) + H(W_0(x,t,y),y),
\end{equation}
if $k=0$, and by
\begin{equation}\label{eq:Psik}
\begin{split}
\Psi_k(x,t,y) &= - \tr (A(y) X_k(x,t,y)) \\
&\quad + \sum_{l=1}^k \frac{1}{l!} \sum_{\substack{i_1+\cdots+i_l=k \\ i_1,\cdots,i_l\geq 1}} B_l(x,t,y)(W_{i_1}(x,t,y),\cdots,W_{i_l}(x,t,y)),
\end{split}
\end{equation}
if $1\leq k\leq m-1$. Using $\Psi_k$, one may rephrase the PDEs \eqref{eq:w1-pde} and \eqref{eq:wk-pde} 
\begin{equation}\label{eq:wk-pde-re}
\p_t w_k (x,t,y) + \Psi_k (x,t,y) = 0 \quad\text{in }\R^n,
\end{equation}
for $0\leq k\leq m-1$. 

Denoting by $T_{m-1}(p_0,p)$ the $(m-1)$-th order Taylor polynomial of $H$ in $p$ at $p_0$, namely, 
\begin{equation*}
T_{m-1}(p_0,p) (y) = \sum_{k=0}^{m-1} \frac{1}{k!} D_p^k H(p_0,y) (p,\cdots,p),
\end{equation*} 
we have
\begin{equation}\label{eq:Tm}
\begin{split}
& T_{m-1} \left( W_0(x,t,y), \sum_{k=1}^m \e^k W_k(x,t,y) \right)(y) - \sum_{k=0}^{m-1} \e^k \tr(A(y) X_k (x,t,y)) \\
& = \sum_{k=0}^{m-1} \e^k \Psi_k(x,t,y)  + \sum_{k=2}^m \sum_{\substack{m\leq i_1+\cdots+i_k\leq km \\ 1\leq i_1,\cdots,i_k\leq m}} \frac{\e^{i_1+\cdots+i_k}}{k!} B_k(x,t,y)(W_{i_1}(x,t,y),\cdots,W_{i_k}(x,t,y)).
\end{split}
\end{equation}
Hence, we apply the Taylor expansion of $H$ in $p$ at $W_0$ up to $(m-1)$-th order and derive that 
\begin{equation}\label{eq:HDetame}
-\e \tr \left( A\left(\frac{x}{\e}\right) D^2 \eta_m^\e(x,t)\right) + H\left( D\eta_m^\e (x,t) ,\frac{x}{\e}\right) = \sum_{k=0}^{m-1} \e^k\Psi_k\left(x,t,\frac{x}{\e}\right) + E_m^\e \left(x,t,\frac{x}{\e}\right),
\end{equation}
where $E_m^\e$ is defined so as to satisfy 
\begin{equation}\label{eq:Eme}
\begin{split}
&E_m^\e (x,t,y) - R_{m-1}\left(W_0(x,t,y),\sum_{k=1}^m\e^k W_k(x,t,y)\right)(y) + \sum_{k=m}^{m+1} \e^k \tr (A(y) X_k(x,t,y)) \\
&=  \sum_{k=2}^m \sum_{\substack{m\leq i_1+\cdots+i_k\leq km \\ 1\leq i_1,\cdots,i_k\leq m}} \frac{\e^{i_1+\cdots+i_k}}{k!} B_k(x,t,y)(W_{i_1}(x,t,y),\cdots,W_{i_k}(x,t,y)),
\end{split}
\end{equation}
with $R_{m-1}(p_0,p)$ being the $(m-1)$-th order remainder term of $H$ in $p$ at $p_0$.

Now using \eqref{eq:wk-pde-re}, we observe that $\eta_m^\e$ solves \eqref{eq:etame-pde} with\begin{equation}\label{eq:psime}
\psi_m^\e(x,t,y) =  \e^m \p_t w_m (x,t,y) + E_m^\e (x,t,y).
\end{equation}
Note that the initial condition of \eqref{eq:etame-pde} is satisfied, due to that of \eqref{eq:eff-HJ} and the assertion (ii) of Lemma \ref{lemma:wk}. Hence, we are only left with proving the estimate \eqref{eq:psime-Linf} for $\psi_m^\e$. 

It is clear that \eqref{eq:wk-Ck-C2a} implies 
\begin{equation}\label{eq:ptwm-Linf}
\norm{ \p_t w_m(x,t,\cdot) }_{L^\infty(\R^n)} \leq C_{m,T},
\end{equation}
for any $T>0$ and any $(x,t)\in \R\times[0,T]$, where $C_{m,T}>0$ is a constant depending only on $n$, $\lambda$, $\Lambda$, $\alpha$, $\alpha'$, $\beta$, $\beta'$, $K$, $L$, $\mu$, $m$ and $T$. On the other hand, using \eqref{eq:H-Ck-C01}, \eqref{eq:w-Ck-C2a} and \eqref{eq:wk-Ck-C2a}, and noting that $\e^{i_1+\cdots+i_k} \leq \e^m$ for any $1\leq i_1,\cdots,i_k\leq m$ satisfying $m\leq i_1+\cdots +i_k\leq km$, we deduce from \eqref{eq:Eme} that
\begin{equation}\label{eq:Eme-Linf}
\norm{ E_m^\e (x,t,\cdot) }_{L^\infty(\R^n)} \leq C_{m,T}\e^m,
\end{equation}
for any $T>0$ and any $(x,t)\in \bar{B}_R\times[0,T]$, with $C_{m,T}>0$ being yet another constant depending only on the same parameters listed above. This finishes the proof. 
\end{proof}

With the aid of Lemma \ref{lemma:wk-cor}, we prove the first main result of this paper.

\begin{theorem}\label{theorem:cr} Suppose that the diffusion coefficient $A$, the Hamiltonian $H$ and the initial data $g$ satisfy \eqref{eq:A-peri} -- \eqref{eq:A-C01}, \eqref{eq:H-peri} -- \eqref{eq:H-Ck-C01}, and respectively \eqref{eq:g-convex} -- \eqref{eq:DHDg-0}. Under these circumstances, let $\{u^\e\}_{\e>0}$ be the sequence of the viscosity solutions of \eqref{eq:HJ}. Then with the viscosity solution $\bar{u}_0$ of \eqref{eq:eff-HJ} and the sequence $\{w_k\}_{k=1}^\infty$ of $k$-th order interior correctors chosen in Lemma \ref{lemma:wk}, we have, for each integer $m\geq 1$, any $0<\e\leq\frac{1}{2}$ and any $T>0$, 
\begin{equation}\label{eq:cr}
\left| u^\e(x,t) - \bar{u}_0(x,t) - \sum_{k=1}^m \e^k w_k \left(x,t,\frac{x}{\e}\right) \right| \leq C_{m,T}\e^m,
\end{equation}
uniformly for all $(x,t)\in\R^n\times [0,T]$, where $C_{m,T}>0$ depends only on $n$, $\lambda$, $\Lambda$, $\alpha$, $\alpha'$, $\beta$, $\beta'$, $K$, $L$, $\mu$, $m$ and $T$. 
\end{theorem}

\begin{proof} The proof follows from Lemma \ref{lemma:wk-cor} and the comparison principle for viscosity solutions. Let $\eta_m^\e$ be as in \eqref{eq:etame}. Due to \eqref{eq:psime-Linf}, we see that $\eta_m^\e + C_{m,T}\e^m t$ and $\eta_m^\e - C_{m,T}\e^m t$ are a viscosity supersolution and, respectively, a viscosity subsolution of \eqref{eq:HJ}. Thus, the comparison principle yields that 
\begin{equation}
|u^\e(x,t) - \eta_m^\e(x,t)| \leq T C_{m,T}\e^m, 
\end{equation}
uniformly for all $(x,t)\in\R^n\times[0,T]$, which finishes the proof. 
\end{proof}

%%%%%%%%%%%%%%%%%%%%%%%%%%%%%%%%%%%%%%%%%%%%%%%%%%%%%%%%%%%%%%%%%%%%%%%%%%%%%%%%%

\section{Generalization to Fully Nonlinear Hamiltonian}\label{section:nl}

In this section, we generalize Theorem \ref{theorem:cr} to the fully nonlinear, viscous Hamilton-Jacobi equation, \eqref{eq:HJ-nl}, whose gradient term is convex and grows quadratically at the infinity. Henceforth, we shall assume that the nonlinear functional $H$ satisfies the following conditions, for any $(M,p,y)\in\sS^n\times\R^n\times\R^n$. 
\begin{enumerate}[(i)]
\item $H$ is periodic in $y$:
\begin{equation}\label{eq:H-peri-nl}
H(M,p,y+k) = H(M,p,y),
\end{equation}
for any $k\in\Z^n$. 
\item $H$ is uniformly elliptic in $M$:
\begin{equation}\label{eq:H-ellip-nl}
\lambda\norm{N} \leq H(M,p,y) - H(M + N,p,y) \leq \Lambda\norm{N},
\end{equation}
for any $N\in\sS^n$ with $N\geq 0$.
\item $H$ has interior $C^{2,\bar\mu}$ estimates: for any $r>0$, any $(M_0,p_0,y_0)\in\sS^n\times\R^n\times\R^n$ with $a= H(M,p,y_0)$, and any $v_0 \in C(\p B_r(y_0))$, there exists a viscosity solution $v\in C(\bar{B}_r(y_0)) \cap C^2( B_r(y_0)) \cap C^{2,\bar\mu}(\bar{B}_{r/2}(y_0))$ of 
\begin{equation*}
\begin{cases}
H( D^2 v + M,p,y_0) = a &\text{in }B_r(y_0),\\
v = v_0 & \text{on }\p B_r(y_0),
\end{cases}
\end{equation*}
such that
\begin{equation*}
\norm{v}_{C^{2,\bar\mu}(\bar{B}_{r/2}(y_0))} \leq K\norm{v_0}_{L^\infty(\p B_r(y_0))}.
\end{equation*}
\item $H$ is convex in $p$:
\begin{equation}\label{eq:H-convex-nl}
H(M,tp+(1-t)q,y) \leq tH(M,p,y) + (1-t)H(M,q,y),
\end{equation}
for any $0\leq t\leq 1$ and any $q\in\R^n$.
\item $H$ has quadratic growth in $p$:
\begin{equation}\label{eq:H-quad-nl}
\alpha|p|^2 - \alpha' \leq H(0,p,y) \leq \beta|p|^2 + \beta'.
\end{equation}
\item $H\in C^\infty(\sS^n\times\R^n;C^{0,1}(\R^n))$ and 
\begin{equation}\label{eq:H-Ck-C01-nl}
\norm{D_M^k D_p^l H(M,p,\cdot)}_{C^{0,1}(\R^n)} \leq K\left( 1 + \norm{M}^{(1-k)_+} + |p|^{(2-l)_+} \right),
\end{equation}
for any pair $(k,l)$ of nonnegative integers.
\end{enumerate}

We shall impose the conditions \eqref{eq:g-convex} -- \eqref{eq:DHDg-0} to the initial data $g$, as we did in the preceding section, once the effective Hamiltonian $\bar{H}$ is determined. The effective Hamiltonian $\bar{H}$ is derived by solving the cell problem \eqref{eq:cell-nl}. Since the proof is analogous to that of Lemma \ref{lemma:cell}, we shall omit the details. Let us remark that although the Hamiltonian $H$ is now fully nonlinear in the Hessian variable $M$, we still have the $C^{2,\mu}$ regularity (for any $0<\mu<\bar\mu$) of periodic viscosity solutions to the cell problem \eqref{eq:cell-nl}, since $H$ has interior $C^{2,\bar\mu}$ estimates for fixed coefficients (assumption (iii)), and it satisfies Lipschitz regularity in $y$ (assumption (vi)). We refer to \cite[Theorem 8.1]{CC}, for details on this regularity theory. Besides, the reader who is unfamiliar with hypothesis (iii) can simply replace it by convexity in $M$. 

\begin{lemma}\label{lemma:cell-nl} For each $p\in\R^n$, there exists a unique real number $\gamma$, for which the following PDE,
\begin{equation}\label{eq:cell-nl}
H(D^2 w, Dw + p, y ) = \gamma \quad\text{in }\R^n,
\end{equation}
has a periodic solution $w\in C^{2,\mu}(\R^n)$, for some $0<\mu<1$ depending only on $n$, $\lambda$ and $\Lambda$. Moreover, $\gamma$ satisfies \eqref{eq:cell-Linf} and, furthermore, a periodic solution $w$ of \eqref{eq:cell-nl} is unique up to an additive constant and satisfies \eqref{eq:cell-C2a}. 
\end{lemma}

As in Section \ref{section:reg}, we shall denote by $\bar{H}$ the effective Hamiltonian of $H$. That is, $\bar{H}:\R^n\ra\R$ is a function defined in such a way that for each $p\in\R^n$, $\bar{H}(p)$ is the unique real number for which the following PDE,
\begin{equation}\label{eq:w-pde-nl}
H(D^2 w, Dw + p, y) =\bar{H}(p) \quad\text{in }\R^n,
\end{equation}
has a periodic viscosity solution in $C^{2,\mu}(\R^n)$. Moreover, we shall also denote by $w:\R^n\times\R^n\ra\R$ by the functional such that for each $p\in\R^n$, $w(p,\cdot) \in C^{2,\mu}(\R^n)$ is the unique periodic solution of \eqref{eq:w-pde-nl} that is normalized so as to satisfy \eqref{eq:w-0}. 

Following the same arguments in their proofs, one may prove that $\bar{H}$ and $w$ satisfy Lemma \ref{lemma:Hb} and Lemma \ref{lemma:w-Linf-C2a}, except for that $w\in C(\R^n;C^{2,\mu}(\R^n))$ for some fixed $0<\mu<1$, rather than any $0<\mu<1$. This is because the proofs of those lemmas do not rely on the linear structure of the diffusion coefficient, but more on its uniform ellipticity. A more important observation is the generalization of Lemma \ref{lemma:w-Ck-C2a}, which amounts to the regularity of $\bar{H}$ and $w$ in the slow variables.

\begin{lemma}\label{lemma:w-Ck-C2a-nl} $\bar{H}\in C^\infty(\R^n)$ and $w\in C^\infty(\R^n;C^{2,\mu}(\R^n))$, for any $0<\mu<\bar\mu$, such that \eqref{eq:w-Ck-C2a} holds, for any $k=0,1,2,\cdots$, any $L>0$ and any $p\in B_L$. 
\end{lemma}

\begin{proof} Let us fix $0<\mu<\bar\mu$. It suffices to prove that $\bar{H}$ and $w$ verify Lemma \ref{lemma:w-C1-C2a}. Moreover, to see this fact, it is enough to show that the linearization argument in the proof of Lemma \ref{lemma:w-C1-C2a} also works out when the Hamiltonian $H$ depends nonlinearly on the Hessian variable $M$. 

Let $w_h$, $\gamma_h$, $W_h$ and $\Gamma_h$ be as in the proof of Lemma \ref{lemma:w-C1-C2a}. Then by linearizing the cell problem \eqref{eq:w-pde-nl} (in both of the Hessian and the gradient variables), we observe that $W_h$ solves
\begin{equation}\label{eq:Wh-pde-nl}
-\tr (A_h(y) D^2 W_h) + B_h(y) \cdot (DW_h + e_k) = \Gamma_h\quad\text{in }\R^n,
\end{equation}
where 
\begin{equation*}
A_h(y) = \int_0^1 - D_M H ( t D_y^2 w_h + (1-t) D_y w_0, D_y w_h + p, y) dt,
\end{equation*}
and 
\begin{equation*}
B_h(y) = \int_0^1 D_p H ( D_y w_0, t D_y w_h + (1-t) D_y w_0 +p + the,y) dt.
\end{equation*}
In comparison of \eqref{eq:Wh-pde-nl} with \eqref{eq:Wh-pde}, one may see that the only major difference here is that the diffusion coefficient, $A_h$, here is not fixed but depends on the parameter $h$.

Nevertheless, $A_h$ is uniformly elliptic not only in $y$ but also in $h$, due to the assumption \eqref{eq:H-ellip-nl}. This implies that Lemma \ref{lemma:cell-lin} is still applicable, and thus $W_h\in C^{2,\mu}(\R^n)$ and satisfies \eqref{eq:Wh-C2a} uniformly for $h$. 

Moreover, since $w$ satisfies \eqref{eq:w-Linf-C2a}, it follows from the regularity assumption \eqref{eq:H-Ck-C01-nl} of $H$ that $A_h\in C^\mu(\R^n)$ and 
\begin{equation}\label{eq:Ah-Ca}
\norm{A_h}_{C^\mu(\R^n)} \leq C,
\end{equation}
where $C>0$ depends only on $n$, $\lambda$ and $\Lambda$. For the same reason, we deduce that $B_h \in C^\mu(\R^n)$ and satisfies \eqref{eq:Bh-Ca}. Furthermore, since $w \in C(\R^n; C^{2,\mu}(\R^n))$, we have $A_h\ra A_0$ and $B_h\ra B_0$ in $C^{\mu'}(\R^n)$, for any $0<\mu'<\mu$, with 
\begin{equation*}
A_0(y) = - D_M H ( D_y w_0, D_y w_0 + p, y),
\end{equation*}
and 
\begin{equation*}
B_0(y) = D_p H ( D_y w_0, D_y w_0 + p ,y).
\end{equation*}

The rest of the proof follows similarly with that of of Lemma \ref{lemma:w-C1-C2a}. In particular, we obtain unique $W_0\in C^{2,\mu}(\R^n)$ and $\Gamma_0\in \R$ such that $W_0$ is the periodic solution  
\begin{equation}\label{eq:W0-pde-nl}
-\tr (A_0(y) D^2 W_0) + B_0(y) \cdot (DW_0 + e_k) = \Gamma_0\quad\text{in }\R^n,
\end{equation}
satisfying $W_0(0) = 0$. We leave out the details to the reader. 
\end{proof} 

Now we are in position to construct the higher order interior correctors of the homogenization problem \eqref{eq:HJ-nl}. We shall now let $g$ satisfy the structure conditions \eqref{eq:g-convex} -- \eqref{eq:DHDg-0}, with $\bar{H}$ being the effective Hamiltonian chosen to satisfy the cell problem \eqref{eq:w-pde-nl}. Next we shall denote by $\bar{u}_0$ the solution of \eqref{eq:eff-HJ}, with the updated data $\bar{H}$ and $g$, and write by $\bar{B}$ the function defined by \eqref{eq:Bb}. Once again, we have $\bar{u}_0\in C^\infty(\R^n\times[0,\infty))$ and $\bar{B}\in C^\infty(\R^n\times[0,\infty))$ with the estimates \eqref{eq:ub0-Ck} and \eqref{eq:Bb-Ck}. 

Let $w_0$, $\{W_k\}_{k=0}^\infty$ and $\{X_k\}_{k=0}^\infty$ denote those defined in \eqref{eq:w0}, \eqref{eq:Wk} and, respectively, \eqref{eq:Xk}, where the sequence $\{w_k\}_{k=1}^\infty$ of higher order interior correctors will be given as below. 

Now that the Hamiltonian $H$ is nonlinear in $M$, we need to apply the Taylor expansion not only in the variable $p$ but also in the variable $M$, in order to obtain the PDEs (or, more precisely, the cell problems) for the higher order interior correctors. In this direction, we consider the coefficient $B_{k,l}$ defined by
\begin{equation}\label{eq:Bkl-nl}
B_{k,l} (x,t,y) = D_M^kD_p^l H(X_0(x,t,y),W_0(x,t,y),y),
\end{equation}
for $k,l=0,1,2,\cdots$. In particular, we shall write 
\begin{equation}\label{eq:A-nl}
A(x,t,y) = -B_{1,0}(x,t,y) = - D_M H(X_0(x,t,y),W_0(x,t,y),y),
\end{equation}
and
\begin{equation}\label{eq:B-nl}
B(x,t,y) = B_{0,1}(x,t,y) = D_p H(X_0(x,t,y),W_0(x,t,y),y).
\end{equation}
Note that $A$ is uniformly elliptic with the same ellipticity bounds as those of $H$. 

\begin{lemma}\label{lemma:wk-nl} Suppose that $H$ and $g$ satisfy \eqref{eq:H-peri-nl} -- \eqref{eq:H-Ck-C01-nl} and, respectively, \eqref{eq:g-convex} -- \eqref{eq:DHDg-0}. Then there exists a sequence $\{w_k\}_{k=1}^\infty$ satisfying the following.
\begin{enumerate}[(i)]
\item $w_k \in C^\infty(\R^n\times[0,\infty); C^{2,\mu}(\R^n))$, for any $0<\mu<\bar\mu$, and satisfies the estimate \eqref{eq:wk-Ck-C2a}, for any $i,j=0,1,2,\cdots$, any $T>0$ and any $(x,t)\in\R^n\times[0,T]$. 
\item $w_k$ is normalized so as to satisfy \eqref{eq:wk-0}. 
\item For each $(x,t)\in\R^n\times(0,\infty)$, $w_k(x,t,\cdot)$ is a periodic solution of
\begin{equation}\label{eq:w1-pde-nl}
\p_t w_0(x,t,y) + H( D_y^2 w_1, D_y w_1 + D_x w_0(x,t,y), y ) = 0 \quad\text{in }\R^n,
\end{equation}
for $k=1$, and 
\begin{equation}\label{eq:wk-pde-nl}
\begin{split}
&\p_t w_{k-1}(x,t,y) - \tr(A(x,t,y) D_y^2 w_k) \\
&+ B(x,t,y) \cdot (D_y w_k + D_x w_{k-1}(x,t,y)) + \Phi_{k-1}(x,t,y) = 0 \quad\text{in }\R^n,  
\end{split}
\end{equation}
for $k\geq 2$, where 
\begin{equation}\label{eq:Phik-nl}
\begin{split}
\Phi_{k-1} (x,t,y) &= -2\tr (A(x,t,y) (D_xD_y w_{k-1} (x,t,y) + D_x^2 w_{k-2} (x,t,y))) \\
&\quad + \sum_{l=2}^{k-1} \frac{1}{l!}  \sum_{\substack{ i_1 + \cdots + i_l = k-1 \\ i_1,\cdots,i_1\geq 1}} \sum_{r=0}^l B_{r,l-r} (x,t,y) ( X_{i_1}(x,t,y),\cdots, X_{i_r} (x,t,y), \\
&\quad\quad\quad\quad\quad\quad\quad\quad\quad \quad\quad\quad\quad\quad\quad\quad W_{i_{r+1}}(x,t,y),\cdots, W_{i_l}(x,t,y))
\end{split}
\end{equation}
with the last summation term understood as zero when $k=2$. 
\end{enumerate}
\end{lemma}

\begin{remark}\label{remark:wk-nl} As mentioned in Remark \ref{remark:wk}, $\Phi_k$ now takes care of the nonlinear effect produced by $H$ in both $M$ and $p$ variables. Moreover, the summation term in the definition \eqref{eq:Phik-nl} of $\Phi_k$ is specifically constructed to have \eqref{eq:HDetame-nl}, by which we will eventually derive the higher order convergence rates for the homogenization problem \eqref{eq:HJ-nl}.
\end{remark}

\begin{proof}[Proof of Lemma \ref{lemma:wk-nl}] The proof follows essentially the same induction argument presented in that of Lemma \ref{lemma:wk}. To avoid any repeating argument, we shall only point out the major difference from the proof of Lemma \ref{lemma:wk}, and ask the reader to fill in the details.  

Here we define $\phi_1$ by \eqref{eq:phi1} with $w$ being the (normalized) periodic solution of \eqref{eq:w-pde-nl} (instead of \eqref{eq:w-pde}), and accordingly set $w_1$ by \eqref{eq:w1} with some $\bar{u}_1$ to be determined. Then we observe that $W_0$ and $X_0$ verify the expressions \eqref{eq:W0} and, respectively, \eqref{eq:X0}. Moreover, we verify that $B_{l,k-l}$ satisfy the estimate \eqref{eq:Bk-Ck-Ca}, for any $l=0,1,\cdots,k$ and any $k=1,2,\cdots$.

The function $f_k$, which takes cares of all the nonlinear effect caused in the $k$-th step of approximation, is now replaced by  
\begin{equation}\label{eq:fk-nl}
\begin{split}
f_k (x,t,y) &= \p_t \tilde{w}_k (x,t,y) + B(x,t,y)\cdot D_x \tilde{w}_k(x,t,y)  \\
&\quad -2\tr ( A(x,t,y) (D_xD_y \tilde{w}_k (x,t,y) + D_x^2 w_{k-1} (x,t,y))) \\
&\quad + \sum_{l=2}^{k-1} \frac{1}{l!}  \sum_{\substack{ i_1 + \cdots + i_l = k-1 \\ i_1,\cdots,i_1\geq 1}} \sum_{r=0}^l B_{r,l-r} (x,t,y) ( X_{i_1}(x,t,y),\cdots, X_{i_r} (x,t,y), \\
&\quad\quad\quad\quad\quad\quad\quad\quad\quad \quad\quad\quad\quad\quad\quad\quad W_{i_{r+1}}(x,t,y),\cdots, W_{i_l}(x,t,y)),
\end{split}
\end{equation}
Due to the periodicity of $f_k$ in $y$, we consider the following cell problem: there exists a unique $\bar{f}_k:\R^n\times(0,\infty)\ra\R^n$ such that for each $(x,t)\in\R^n\times(0,\infty)$, the PDE,
\begin{equation}\label{eq:phik-pde-nl}
-\tr (A(x,t,y) D_y^2 \phi_{k+1}) + B(x,t,y) \cdot D_y \phi_{k+1} + f_k(x,t,y) = \bar{f}_k(x,t) \quad\text{in }\R^n,
\end{equation}
has a periodic viscosity solution. The rest of the proof can be derived by following that of Lemma \ref{lemma:wk}, whence we omit the details. 
\end{proof}

The next lemma is the corresponding version of Lemma \ref{lemma:wk-cor} for fully nonlinear Hamiltonian $H$.

\begin{lemma}\label{lemma:wk-cor-nl}
Let $\{w_k\}_{k=1}^\infty$ be chosen as in Lemma \ref{lemma:wk-nl}. Then for each integer $m\geq 1$ and each $0<\e\leq\frac{1}{2}$, the function $\eta_m^\e$, defined by \eqref{eq:etame}, is a viscosity solution of 
\begin{equation}\label{eq:etame-pde-nl}
\begin{dcases}
\p_t \eta_m^\e + H\left( \e D^2 \eta_m^\e, D \eta_m^\e, \frac{x}{\e}\right) = \psi_m^\e\left(x,t,\frac{x}{\e}\right) & \text{in }\R^n\times(0,\infty),\\
\eta_m^\e = g & \text{on }\R^n\times\{t=0\},
\end{dcases}
\end{equation}
where $\psi_m^\e \in C(\R^n\times[0,\infty);L^\infty(\R^n))$ satisfies \eqref{eq:psime-Linf}, for any $T>0$ and all $(x,t)\in \R^n\times[0,T]$.
\end{lemma}

\begin{proof} As in the proof of Lemma \ref{lemma:wk-nl}, we shall mention the key points that need to be modified from the proof of Lemma \ref{lemma:wk-cor}, in order to take care of the nonlinear effect in the Hessian variable of $H$. Let us begin by fixing $m\geq 1$ and $0<\e\leq\frac{1}{2}$, and replacing $w_{m+1}$ and $w_{m+2}$ by the identically zero functions, again for the notational convenience.

We shall define $\Psi_k$ by 
\begin{equation}\label{eq:Psi0-nl}
\Psi_0(x,t,y) = H( X_0(x,t,y), W_0(x,t,y), y),
\end{equation}
if $k=0$, and by
\begin{equation}\label{eq:Psik-nl}
\begin{split}
\Psi_k(x,t,y) &= \sum_{l=1}^{k-1} \frac{1}{l!}  \sum_{\substack{ i_1 + \cdots + i_l = k-1 \\ i_1,\cdots,i_1\geq 1}} \sum_{r=0}^l B_{r,l-r} (x,t,y) ( X_{i_1}(x,t,y),\cdots, X_{i_r} (x,t,y), \\
&\quad\quad\quad\quad\quad\quad\quad\quad\quad \quad\quad\quad\quad\quad\quad W_{i_{r+1}}(x,t,y),\cdots, W_{i_l}(x,t,y))
\end{split}
\end{equation}
if $1\leq k\leq m-1$. Then it follows from the PDEs \eqref{eq:w1-pde-nl} and \eqref{eq:wk-pde-nl} that \eqref{eq:wk-pde-re} holds for $0\leq k\leq m-1$.

Applying the Taylor expansion of $H$ in $(M,p)$ at $(X_0,W_0)$ up to $(m-1)$-th  order, and after some calculations similar to those in \eqref{eq:Tm}, we obtain that
\begin{equation}\label{eq:HDetame-nl}
H \left( \e D^2 \eta_m^\e (x,t), D\eta_m^\e(x,t),\frac{x}{\e}\right) = \sum_{k=0}^{m-1} \e^k \Psi_k\left(x,t,\frac{x}{\e}\right) + E_m^\e \left(x,t,\frac{x}{\e}\right),
\end{equation}
where $E_m^\e$ is defined so as to satisfy
\begin{equation}\label{eq:Eme-nl}
\begin{split}
&E_m^\e (x,t,y) - R_{m-1}\left((X_0(x,t,y),W_0(x,t,y)),\left(\sum_{k=1}^{m+1} \e^k X_k(x,t,y), \sum_{k=1}^m\e^k W_k(x,t,y)\right)\right)(y)\\
&= \sum_{k=2}^m \sum_{\substack{m\leq i_1+\cdots+i_k\leq km \\ 1\leq i_1,\cdots,i_k\leq m}} \frac{\e^{i_1+\cdots+i_k}}{k!} \sum_{l=0}^k B_{l,k-l}(x,t,y)(X_{i_1}(x,t,y)\cdots,X_{i_l}(x,t,y), \\
&\quad\quad\quad\quad\quad\quad\quad\quad\quad \quad\quad\quad\quad\quad\quad\quad\quad\quad W_{i_{l+1}}(x,t,y),\cdots,W_{i_k}(x,t,y)),
\end{split}
\end{equation}
where $R_{m-1}((M_0,p_0),(M,p))$ denotes the $(m-1)$-th order remainder term of $H$ in $(M,p)$ at $(M_0,p_0)$.

We deduce from \eqref{eq:HDetame-nl} that $\eta_m^\e$ solves \eqref{eq:etame-pde-nl} with $\psi_m^\e$ defined by \eqref{eq:psime}. The rest of the proof follows similarly to that of Lemma \ref{lemma:wk-cor}. In particular, we have \eqref{eq:Eme-Linf}, since $B_{l,k-l}$ and $w_k$ satisfy the estimate \eqref{eq:Bk-Ck-Ca} and, respectively, \eqref{eq:wk-Ck-C2a}. We leave out the details to the reader.  
\end{proof}

Finally, we generalize Theorem \ref{theorem:cr} to the regime of fully nonlinear, viscous Hamilton-Jacobi equation, as stated below.

\begin{theorem}\label{theorem:cr-nl} Suppose that the Hamiltonian $H$ and the initial data $g$ satisfy \eqref{eq:H-peri-nl} -- \eqref{eq:H-Ck-C01-nl} and, respectively, \eqref{eq:g-convex} -- \eqref{eq:DHDg-0}. Under these circumstances, let $\{u^\e\}_{\e>0}$ be the sequence of the viscosity solutions of \eqref{eq:HJ-nl}. Then with the viscosity solution $\bar{u}_0$ of \eqref{eq:eff-HJ} and the sequence $\{w_k\}_{k=1}^\infty$ of $k$-th order interior correctors chosen in Lemma \ref{lemma:wk-nl}, we have, for each integer $m\geq 1$, any $0<\e\leq\frac{1}{2}$ and any $T>0$, 
\begin{equation}\label{eq:cr-nl}
\left| u^\e(x,t) - \bar{u}_0(x,t) - \sum_{k=1}^m \e^k w_k \left(x,t,\frac{x}{\e}\right) \right| \leq C_{m,T}\e^m,
\end{equation}
uniformly for all $(x,t)\in\R^n\times [0,T]$, where $C_{m,T}>0$ depends only on $n$, $\lambda$, $\Lambda$, $\alpha$, $\alpha'$, $\beta$, $\beta'$, $K$, $L$, $\mu$, $\bar\mu$, $m$ and $T$. 
\end{theorem}

\begin{proof} The proof follows the same comparison argument as that in the proof of Theorem \ref{theorem:cr}. Let $\eta_m^\e$ be as in Lemma \ref{lemma:wk-cor-nl}. According to Lemma \ref{lemma:wk-cor-nl}, $\eta_m^\e + C_{m,T}\e^m t$ and $\eta_m^\e - C_{m,T}\e^m t$ are a viscosity supersolution and, respectively, a viscosity subsolution of \eqref{eq:HJ-nl}, for some constant $C_{m,T}>0$ depending only on $n$, $\lambda$, $\Lambda$, $\alpha$, $\alpha'$, $\beta$, $\beta'$, $K$, $L$, $\mu$, $\bar\mu$, $m$ and $T$. Therefore, the comparison principle yields that 
\begin{equation}
|u^\e(x,t) - \eta_m^\e(x,t)| \leq T C_{m,T}\e^m, 
\end{equation}
uniformly for all $(x,t)\in\R^n\times[0,T]$. This completes the proof. 
\end{proof}

\end{document}